\documentclass{birkjour}
\usepackage{amsmath,amsthm,amssymb,enumerate,esint,cases,fixmath,tikz}
\usepackage{nccmath,amscd,latexsym}
\usetikzlibrary{decorations.pathreplacing}
\usepackage[bookmarksopen=true]{hyperref}
\newtheorem{theorem}{Theorem}
\newtheorem{lemma}[theorem]{Lemma}

\theoremstyle{definition}

\newtheorem{remark}[theorem]{Remark}
\theoremstyle{remark}

\numberwithin{equation}{section}
\numberwithin{theorem}{section}
\newcounter{smalllist}

\newcommand{\MRhref}[2]{\href{http://www.ams.org/mathscinet-getitem?mr=#1}{#2}}
\renewcommand{\MR}[1]{\MRhref{#1}{MR #1}}
\newcommand{\ds}{\displaystyle}

\newcommand{\bi}{\bibitem}
\newcommand{\no}{\notag}
\newcommand{\lb}{\label}
\newcommand{\f}{\frac}
\newcommand{\bb}{\mathbb}
\newcommand{\cc}{\mathcal}
\newcommand{\bs}{\backslash}

\newcommand{\ti}{\tilde}

\newcommand{\pd}{\partial}

\newcommand{\si}{\sigma}
\newcommand{\la}{\lambda}

\newcommand{\al}{\alpha}
\newcommand{\be}{\beta}
\newcommand{\Ga}{\Gamma}
\newcommand{\ga}{\gamma}
\newcommand{\De}{\Delta}
\newcommand{\de}{\delta}

\newcommand{\eps}{\varepsilon}
\newcommand{\E}{\mathsf{E}}
\newcommand{\ca}{\text{\rm{cap}}}
\newcommand{\tr}{\text{\rm{tr}}}
\newcommand{\dist}{\text{\rm{dist}}}
\newcommand{\diam}{\text{\rm{diam}}}
\newcommand{\lip}{\text{\rm{Lip}}}
\newcommand{\ess}{{ess}}
\newcommand{\dd}{{d}}

\renewcommand{\Re}{\text{\rm Re}}
\renewcommand{\Im}{\text{\rm Im}}
\renewcommand{\oint}{\varointctrclockwise}
\newcommand{\bsn}{\raisebox{1pt}{$\smallsetminus$}}
\begin{document}
\title[Lieb--Thirring Inequalities for Jacobi Matrices]{Lieb--Thirring Inequalities for Finite and Infinite Gap Jacobi Matrices}
\author[J.~S.~Christiansen]{Jacob S.~Christiansen}
\address{
Centre for Mathematical Sciences\\
Lund University, Box 118\\
SE-22100, Lund, Sweden.}
\email{stordal@maths.lth.se}

\author[M.~Zinchenko]{Maxim Zinchenko}
\address{
Department of Mathematics and Statistics\\
University of New Mexico\\
Albuquerque, NM 87131.}
\email{maxim@math.unm.edu}

\thanks{JSC is supported in part by the Research Project Grant DFF--4181-00502 from the Danish Council for Independent Research.\\
MZ is supported in part by Simons Foundation Grant CGM-281971.}

\date{\today}

%%%%%%%%%%%%%%%%%%%%%%%%%%%%%%%%%%%%%%%
\begin{abstract}
We establish Lieb--Thirring power bounds on discrete eigenvalues of Jacobi operators for Schatten class perturbations under very general assumptions.
Our results apply, in particular, to perturbations of reflectionless Jacobi operators with finite gap and Cantor-type essential spectrum.
\end{abstract}
%%%%%%%%%%%%%%%%%%%%%%%%%%%%%%%%%%%%%%%

\maketitle

%%%%%%%%%%%%%%%%%%%%%%%%%%%%%%%%%%%%%%%
\section{Introduction}
%%%%%%%%%%%%%%%%%%%%%%%%%%%%%%%%%%%%%%%

Let $A$ be a self-adjoint operator on some Hilbert space $\mathcal{H}$ and define
\begin{equation}
S^p(A)=\sum_{\la\in\si_\dd(A)} \dist\bigl(\la,\si_\ess(A)\bigr)^p, \quad p\geq 0,
\end{equation}
where $\si_\dd$ is the discrete and $\si_\ess$ the essential spectrum.
Each term in the sum is repeated according to the multiplicity of the eigenvalue $\la$.
Upper bounds on $S^p(A)$ for various choices of $A$ and values of $p$ have shown to be useful in studies of
quantum mechanics, differential equations, and dynamical systems.
The reader is referred to, e.g., \cite{HM07} for history and reviews. %of such bounds.

The original {\it Lieb--Thirring inequalities} deal with perturbations of the Laplacian on $L^2(\bb R^d)$ and assert that
\begin{equation}
\label{LT Laplace}
 S^p(-\De+V)\leq L_{p,d} \int_{\bb R^d} V_-(x)^{p+d/2}dx,
\end{equation}
where $V_-=\max\{0,-V\}$ and $L_{p,d}$ is a constant independent of $V$. This was proved by Lieb and Thirring in 1976 for $p>1/2$ if $d=1$ and for $p>0$ if $d\geq 2$.
Their motivation was a rigorous proof of the stability of matter, see \cite{LT75, LT76}. When $d=1$, the bound in \eqref{LT Laplace} fails to hold for $p<1/2$ and
the endpoint result for $p=1/2$ was proved by Weidl \cite{We96} some 20 years later. %, is called the {\it critical bound}.

%The endpoint result for $p=0$ if $d\geq 3$ was proven independently by Cwikel \cite{CWi77}, Lieb \cite{Li76}, and Rosenblum \cite{Ro76}.
%with an alternate proof and optimal constant due to Hundertmark, Lieb, and Thomas \cite{HLT98}.

In this paper, we consider self-adjoint Jacobi operators on $\ell^2(\bb Z)$ represented by the tridiagonal Jacobi matrices
\begin{align}
J=\begin{pmatrix}
\ddots&\ddots&\ddots\\
& a_0 & b_1 & a_1 &     &\\
&     & a_1 & b_2 & a_2 &\\
&     &     & a_2 & b_3 & a_3\\
&&&&\ddots&\ddots&\ddots
\end{pmatrix}
\end{align}
with bounded parameters $a_n>0$ and $b_n\in\bb R$.
Our main goal is to obtain Lieb--Thirring inequalities for perturbations of almost periodic Jacobi matrices. In the general setting of almost periodic parameters,
the spectrum is typically a Cantor set.
We are motivated by the recent developments in spectral theory of Jacobi matrices, see \cite{JSC1,JSC2,DGL16,DGSV}, and in particular by the finite gap results of
Frank and Simon \cite{FS11} and also Hundertmark and Simon \cite{HS08}.

%%%%%%%%%%%%%%%%%%%%%%%%%%%%%%%%%%%%%%%%%
Before explaining our new results, let us briefly go through what is already known.
%and discuss some consequences of Lieb--Thirring inequalities for Jacobi matrices.
The spectral theory for perturbations of the {\it free} Jacobi matrix, $J_0$, (i.e., the case of $a_n\equiv 1$ and $b_n\equiv 0$)
is well understood and developed in much detail, see \cite{Si11}.
When $J=\{a_n, b_n\}_{n=1}^\infty$ is a compact perturbation of $J_0$, Hundertmark and Simon \cite{HS02} proved that
\begin{equation}
\label{LT Jacobi}
% \sum_{\la\in\si_\dd(J)} \dist\bigl(\la, [-2,2] \bigr)^p\leq
 S^p(J)\leq L_{p,\,J_0} \sum_{n=1}^\infty 4\vert a_n-1\vert^{p+1/2}+\vert b_n\vert^{p+1/2}, \quad p\geq 1/2,
\end{equation}
with some explicit constants $L_{p,\,J_0}$ that are independent of $J$. As in the continuous case, the inequality %\eqref{LT Jacobi}
is false for $p<1/2$.
%The case $p=1/2$, called the critical bound, is of interest
%In particular, if $J-J_0$ is trace class, then $S^{1/2}(J)$
%(i.e., the total distance of the eigenvalues to the interval $[-2, 2]$)
%is bounded by the trace norm of the perturbation. This result is also interesting
%from an orthogonal polynomial point of view. In combination with Szeg\H{o}'s theorem (see, e.g., \cite[Chap.~3]{Si11}) it implies that any trace class
%perturbation of $J_0$ remains within the so-called Szeg\H{o} class. This was conjectured by Nevai \cite{Ne79} and proven years later by Killip and Simon \cite{KS03}.
%More recently, Damanik, Killip, and Simon \cite{DKS10} proved Nevai's conjecture for perturbations of periodic Jacobi matrices. Shortly thereafter, this was extended by %Frank and Simon
More recently, the $p=1/2$ case of \eqref{LT Jacobi} was extended
to finite gap Jacobi matrices in \cite{DKS10,FS11,HS08}.
In the setting of periodic and almost periodic parameters, the role of $J_0$ as a natural limiting point is taken over by the so-called isospectral torus,
denoted $\cc T_\E$. See, e.g., \cite{JSC2,CSZ1,SY97} for a deeper discussion of this object. % for sets of Parreau--Widom type.
The finite gap version of \eqref{LT Jacobi} with $p=1/2$ says that if $\E$ is a finite gap set (i.e., a finite union of disjoint, compact intervals) and $J$ is a trace class perturbation of an element $J'=\{a_n', b_n'\}_{n=-\infty}^\infty$ in $\cc T_\E$, then
\begin{equation}
\label{LT fg}
% \sum_{\la\in\si_\dd(J)} \dist\bigl(\la, \E \bigr)^{1/2}\leq
S^{1/2}(J) \leq L_{1/2,\,\E} \sum_{n=-\infty}^\infty \vert a_n-a_n'\vert+\vert b_n-b_n'\vert.
\end{equation}
As before, the constant $L_{1/2,\,\E}$ is independent of $J$, $J'$ and only depends on the underlying set $\E$.
In comparison with previous attempts, the novelty of \cite{FS11} lies in a clever reduction of the Lieb--Thirring bound for eigenvalues in a single gap to the
previously known case of no gaps. However, the method yields little information about the constants that come with each gap. As a result, this approach is hard
to generalize to sets with infinitely many gaps.
%finding a clever way to deal with eigenvalues in the gaps of $\E$. However, since the estimation is carried out gap by gap, it seems very hard to generalize their methods to sets with infinitely many gaps.

%%%%%%%%%%%%%%%%%%%%%%%%%%%%%%%%%%%%%%%%%%
In the present paper, we improve and extend the eigenvalue bounds of \cite{HS08} to infinite gap Jacobi matrices and obtain Lieb--Thirring bounds for Schatten class perturbations (i.e., non trace class perturbations) of finite and infinite gap matrices.
Our new abstract results can be described in the following way.
%Let $\E\subset\bb R$ be a compact set and suppose $J'=\{a_n', b_n'\}_{n=-\infty}^\infty$ is a two-sided Jacobi matrix with $\si(J')=\si_\ess(J)=\E$. While compact perturbations of $J'$ do not change the essential spectrum, they usually produce a number of discrete eigenvalues outside of $\E$. Due to a general result of Kato \cite{Ka87}, we have
%\begin{equation}
%\label{Kato}
%\sum_{\la\in\si_\dd(J)} \dist\bigl(\la, \E \bigr)\leq \|\de J\|_1
%\end{equation}
%for any perturbation $J=J'+\de J$. Here, $\|\cdot\|_1$ denotes the trace norm and \eqref{Kato} is in fact optimal for large perturbations. Towards the other end, one can never do better than the critical bound which in turn is optimal for small perturbations (cf.~\cite{HS02}).
Let $J'$ be a two-sided Jacobi matrix with $\si(J')=\si_\ess(J')$ and suppose $J=J'+\de J$ is a compact perturbation of $J'$. While compact perturbations do not
change the essential spectrum, they usually produce a number of discrete eigenvalues. By a general result of Kato \cite{Ka87} specialized to the present setting,
we have the following bound
\begin{align} \lb{Kato}
S^1(J) \leq \|\de J\|_1 \leq \sum_{n=-\infty}^\infty 4|\de a_n|+|\de b_n|,
\end{align}
where $\|\cdot\|_1$ denotes the trace norm. In contrast to the Lieb--Thirring bounds, the power on the eigenvalues in \eqref{Kato} is the same as on the perturbation.
Kato's inequality is optimal for perturbations with large sup norm. On the other hand, the Lieb--Thirring bound with $p=1/2$ is optimal for perturbations with small sup
norm (cf.~\cite{HS02}).
Our first main result (Theorem \ref{LTthm1}) in Section \ref{sec3} can be thought of as an interpolation between Kato's bound \eqref{Kato} and the Lieb--Thirring bound
\eqref{LT fg}. More precisely, we show that under certain assumptions on the unperturbed matrix $J'$, a Lieb--Thirring bound of the form
\begin{equation}
\lb{LT trace class}
S^p(J)\leq L_{p,\,J'}\sum_{n=-\infty}^\infty 4|\de a_n|+|\de b_n|, \quad 1/2<p<1,
\end{equation}
holds for any trace class perturbation $J$. The constant $L_{p,\,J'}$ is independent of $\de J$ and can be specified explicitly.
Our second main result (Theorem \ref{LTthm2}) is more general, but has slightly stronger assumptions on $J'$.
We show that
\begin{equation}
\label{LT schatten}
S^p(J)\leq  L_{p,\,J'}\sum_{n=-\infty}^\infty 4|\de a_n|^{p+1/2}+|\de b_n|^{p+1/2}, \quad p>1/2,
\end{equation}
whenever $\de J=J-J'$ belongs to the Schatten class $\mathcal{S}_{p+1/2}$. As before, the explicit constant $L_{p,\,J'}$ does not depend on $\de J$.
We mention in passing that for trace class perturbations and $1/2<p<1$, one has both \eqref{LT trace class} and \eqref{LT schatten} since $\cc S_1\subset\cc S_{p+1/2}$.
The latter bound is slightly better for small perturbations.

%%%%%%%%%%%%%%%%%%%%%%%%%%%%%%%%%%%%%%%%%%%%%%%
As for the classical Lieb--Thirring bounds, our proofs of \eqref{LT trace class} and \eqref{LT schatten} rely on a version of the Birman--Schwinger principle and a new estimate for
\begin{equation}
\label{trace}
\Vert D^{1/2}(J'-x)^{-1}D^{1/2} \Vert_1
\end{equation}
with $D\geq 0$ being a diagonal matrix. We establish the latter in Section \ref{sec2}.
Using the functional calculus, one can express the positive and negative parts of $(J'-x)^{-1}$ as Cauchy-type integrals.
This fact enables us (see Theorem \ref{TNEthm}) to give an upper bound on \eqref{trace} in terms of $\Vert D\Vert_1$ and a slight variation of the $m$-functions for the
spectral measures $d\rho_n$ of $(J', \de_n)$. To estimate further, we impose absolute continuity of $d\rho_n$ and the reflectionless condition (to be defined in Section
\ref{sec2}). If $\E$ is a homogeneous set in the sense of Carleson \cite{Carl}
(i.e., there is an $\eps>0$ so that $\vert(x-\de, x+\de)\cap{\E}\vert\geq\de\eps$ for all $x\in\E$ and all $\de<\diam(\E)$), then both conditions are fulfilled for every $J'$ in the isospectral torus $\cc T_\E$.
Theorem \ref{ReflEstThm} then gives an upper bound that only involves the ordinary $m$-function, but for all reflectionless measures on $\E$.
This result is the key to our Lieb--Thirring bounds. % of Section \ref{sec4}.

%%%%%%%%%%%%%%%%%%%%%%%%%%%%%%%%%%%%%%%%%%%%%%%%
The second part of the paper focuses on explicit examples of infinite gap sets for which our results apply. This has so far been unexplored territory
although the issue is quite natural from an almost periodic point of view.
In Section \ref{sec4}, a detailed study of infinite band sets with one accumulation point is followed by a thorough investigation of fat Cantor sets.
For both types of structure, which are defined from a sequence $\{\eps_k\}_{k=1}^\infty$ with $0<\eps_k<1$, we obtain Lieb--Thirring bounds as
in \eqref{LT trace class}--\eqref{LT schatten} for perturbations of Jacobi matrices from the isospectral tori.
This is done under various assumptions on $\{\eps_k\}_{k=1}^\infty$ in Theorems \ref{InfBandLT} and \ref{CantorLT}.
A typical result in this direction is \eqref{LT trace class} for perturbations of $J'\in\cc T_\E$, where $\E$ is an infinite band set with parameters
$\{\eps_k\}_{k=1}^\infty$ satisfying $\sum_{k=1}^\infty \eps_k<\infty$. The summability condition in question is nearly optimal as it is, in fact,
a necessary condition for the Lieb--Thirring bound in the case $p=1/2$.

We also provide alternative versions of our bounds where the distance to the essential spectrum is measured by the potential theoretic Green function.
Since the infinite gap sets discussed in Section~\ref{sec4} are homogeneous, and hence, regular for potential theory, the Green function $g$ is the unique continuous
function which is positive and harmonic in $\bb C\bsn \E$, vanishes on $\E$, and for which $g(z)-\log|z|$ is harmonic at $\infty$.
Our alternative Lieb--Thirring bounds hold for $J=J'+\de J$ with $J'$ from the isospectral torus $\cc T_\E$ and take the form
\begin{equation}
\label{alt1}
\sum_{\la\in\si(J)\bs\E} g(\la)^{p} \leq  L_{p,\,\E}\sum_{n=-\infty}^\infty |\de a_n|^{(p+1)/2}+|\de b_n|^{(p+1)/2},
\quad p>1,
\end{equation}
where the constant $L_{p,\,\E}$ is independent of $J$, $J'$ and only depends on $p$ and the underlying set $\E$. In the case of an infinite band set,
a sufficient condition for \eqref{alt1} is $\sum_{k=1}^\infty \eps_k<\infty$. This, in turn, is shown to be a necessary condition for the alternative
bound \eqref{alt1} in the case $p=1$. For the middle $\mathbold{\eps}$-Cantor sets of Section \ref{sec4b}, a stronger condition seems to be needed and we show that \eqref{alt1} is satisfied provided $\eps_k\leq C/2^k$ for all large $k$.
%%%%%%%%%%%%%%%%%%%%%%%%%%%%%%%%%%%%%%%

%%%%%%%%%%%%%%%%%%%%%%%%%%%%%%%%%%%%%%%
\section{Trace Norm Estimates}
\label{sec2}
%%%%%%%%%%%%%%%%%%%%%%%%%%%%%%%%%%%%%%%

In this section, we obtain trace norm estimates which will play a crucial role in the proofs of our main results.
%%%%%%%%%%%%%%%%%%%%%%%%%%%%%%%%%%%%%%%
\begin{theorem} \lb{TNEthm}
Suppose $D\geq0$ is a diagonal matrix of trace class and $J'$ is a self-adjoint Jacobi matrix. Let $\E=\si(J')$, then
\begin{align} \lb{TNE0}
\|D^{1/2}(J'-x)^{-1}D^{1/2}\|_1 \leq \|D\|_1\,\sup_{n\in\bb Z} \int_\E \f{d\rho_n(t)}{|t-x|}, \quad x\in\bb R\bsn\E,
\end{align}
where $d\rho_n$ is the spectral measure of $(J',\de_n)$, that is, the measure from the Herglotz representation of the $n$th diagonal entry of $(J'-z)^{-1}$,
\begin{align}
\bigl\langle\de_n,(J'-z)^{-1}\de_n\bigr\rangle = \int_\E \f{d\rho_n(t)}{t-z}, \quad z\in\bb C\bsn\E.
\end{align}
\end{theorem}
%%%%%%%%%%%%%%%%%%%%%%%%%%%%%%%%%%%%%%%
\begin{proof}
Fix $x\in\bb R\bsn \E$ and let $\E_\pm=\E\cap(x,\pm\infty)$. In addition, let $R_\pm$ be the positive and negative parts of $(J'-x)^{-1}$ defined by
\begin{align} \lb{Rpm}
R_\pm = \pm P_{\E_\pm}(J') (J'-x)^{-1} P_{\E_\pm}(J'),
\end{align}
where $P_{\E_\pm}(J')$ are the spectral projections of $J'$ onto the sets $\E_\pm$. Then
\begin{equation}
(J'-x)^{-1}=R_+-R_-, \quad R_\pm\geq0,
\end{equation}
and hence, $D^{1/2}R_\pm D^{1/2}\geq0$. This yields the trace norm estimate,
\begin{align} \lb{TNE1}
\|D^{1/2}(J'-x)^{-1}D^{1/2}\|_1
&= \|D^{1/2}(R_+-R_-)D^{1/2}\|_1 \no
\\
&\leq \|D^{1/2}R_+D^{1/2}\|_1+\|D^{1/2}R_-D^{1/2}\|_1
\\
&= \tr\big[D^{1/2}R_+D^{1/2}\big] + \tr\big[D^{1/2}R_-D^{1/2}\big]. \no
\end{align}
Let $\Ga_\pm$ be non-intersecting rectangular contours around $\E_\pm$. Using the functional calculus we can express the RHS of \eqref{Rpm} as a Cauchy-type integral,
\begin{align}
R_\pm = \f{\pm1}{2\pi i} \oint_{\Ga_\pm} \f{1}{z-x}(z-J')^{-1}dz.
\end{align}
Multiplying by $D^{1/2}$ from the left and from the right and taking the trace then give
\begin{align}
\tr\big[D^{1/2}R_\pm D^{1/2}\big] &= \f{\pm1}{2\pi i} \oint_{\Ga_\pm} \f{1}{z-x}\tr\big[D^{1/2}(z-J')^{-1}D^{1/2}\big]dz
\\
&= \f{\pm1}{2\pi i} \sum_{n\in\bb Z} \langle\de_n,D\de_n\rangle \oint_{\Ga_\pm} \f{1}{z-x} \bigl\langle\de_n,(z-J')^{-1}\de_n\bigr\rangle dz. \no
\end{align}
Finally, deforming the contours $\Ga_\pm$ into $\E_\pm$ traversed twice in the opposite directions and noting that
\begin{align}
& \frac{1}{2\pi i}\Bigl(\bigl\langle\de_n,(t-i\eps-J')^{-1}\de_n\bigr\rangle-\bigl\langle\de_n,(t+i\eps-J')^{-1}\de_n\bigr\rangle \Bigr) \no
\\ &\quad = \frac{1}{\pi}\Im\bigl\langle\de_n,(J'-t-i\eps)^{-1}\de_n\bigr\rangle \xrightarrow{\;\;w\;\;} d\rho_n(t)
\; \mbox{ as } \; \eps\to 0^+,
\end{align}
%and $\f{1}{\pi}\Im\langle\de_n,(J'-t-i\eps)^{-1}\de_n\rangle dt$ converges weakly to $d\rho_n(t)$ as $\eps\to0^+$,
we obtain
\begin{align} \lb{TNE2}
&\tr[D^{1/2}R_\pm D^{1/2}] =
\sum_{n\in\bb Z}\langle\de_n,D\de_n\rangle \int_{\E_\pm} \f{d\rho_n(t)}{|t-x|}.
\end{align}
Combining \eqref{TNE2} with \eqref{TNE1} yields \eqref{TNE0}.
\end{proof}
%%%%%%%%%%%%%%%%%%%%%%%%%%%%%%%%%%%%%%%

A natural question is how to estimate the integrals in \eqref{TNE0}, but first some notation.
%As a final preparation we set up the notation to be used throughout the paper.
Throughout the paper, $\E\subset \bb R$ will denote a compact set. We let $\be_0=\inf\E$ and $\al_0=\sup\E$. Since $[\be_0,\al_0]\bsn\E$ is an open set,
it can be written as a disjoint union of open intervals; hence,
\begin{align} \lb{SetE}
\E=\bigl[\be_0,\al_0\bigr]\bsn\bigcup_{j\geq1}\bigl(\al_j,\be_j\bigr).
\end{align}
For convenience, we define $(\al,\be)$ with $\be<\al$ by
\begin{equation}
(\al,\be)=(-\infty,\be)\cup(\al,\infty).
\end{equation}
With this convention, we shall refer to $(\al_j,\be_j)$, $j\geq0$, as the {\it gaps} of $\E$.
We also call $(\al_j, \be_j)$, $j\geq 1$, the inner gaps and $(\al_0,\be_0)$ the outer gap of $\E$.
%It is convenient to have a common notion for the `inner' gaps $(\al_j, \be_j)$, $j\geq 1$, and the `outer' gap $(\al_0,\be_0)$.

For a probability measure $d\rho$ supported on $\E$, define the associated Herglotz function by
\begin{equation}
m(z)=\int_\E\f{d\rho(t)}{t-z}, \quad z\in\bb C\bsn\E.
\end{equation}
The measure $d\rho$ is called {\it reflectionless} (on $\E$) if
\begin{equation}
\Re[m(x+i0)]=0 \; \mbox{ for a.e.\ $x\in\E$}.
\end{equation}
When $\E$ is essentially closed (i.e., $|\E\cap(x-\eps,x+\eps)|>0$ for all $x\in\E$ and $\eps>0$), we will denote the set of all reflectionless probability measures supported on $\E$ by $\cc R_\E$.
Reflectionless measures appear prominently in spectral theory of finite and infinite gap Jacobi matrices (see, e.g., \cite{JSC2,Re11,Si11,SY97}).
In particular, the isospectral torus $\cc T_\E$ associated with %a compact set
$\E$ is the set of all Jacobi matrices $J'$ that are reflectionless on $\E$
(i.e., the spectral measure of $(J',\de_n)$ belongs to $\cc R_\E$ for every $n\in\bb Z$) and for which $\si(J')=\E$.
It is well known (see for example \cite{SY97}) that $d\rho$ is a reflectionless probability measure on $\E$ if and only if $m(z)$ is of the form
\begin{align}
\label{m fct}
m(z) = \f{-1}{\sqrt{(z-\be_0)(z-\al_0)}} \prod_{j\geq1}\f{z-\ga_j}{\sqrt{(z-\al_j)(z-\be_j)}},
\end{align}
for some $\ga_j\in[\al_j,\be_j]$, $j\geq1$.

For absolutely continuous reflectionless measures we have the following upper bound \eqref{ReflEst1} for the integrals that appear on the RHS of
our trace norm estimate \eqref{TNE0}. This result is the key to our Lieb--Thirring bounds for perturbations of reflectionless Jacobi matrices in Section \ref{sec4}.
%%%%%%%%%%%%%%%%%%%%%%%%%%%%%%%%%%%%%%%
\begin{theorem} \lb{ReflEstThm}
Let $\E\subset\bb R$ be an essentially closed compact set and suppose $d\rho$ is a reflectionless absolutely continuous probability measure on $\E$. Denote the gaps of $\E$ as in \eqref{SetE}. Then, for every $k\geq1$,
\begin{align} \lb{ReflEst1}
\int_\E \f{d\rho(t)}{|t-x|} \leq C_k \sup_{d\mu\in\cc R_\E}\left|\int_\E \f{d\mu(t)}{t-x}\right|, \quad x\in(\al_k,\be_k),
\end{align}
where
\begin{align}
C_k = 9+2\min\left\{\log\f{\be_k-\be_0}{\be_k-\al_k},\, \log\f{\al_0-\al_k}{\be_k-\al_k}\right\}.
\end{align}
Equivalently, if for fixed $x\in(\al_k,\be_k)$ we define $\ti\ga_j\in\{\al_j,\be_j\}$ such that
\begin{equation}
|x-\ti\ga_j|=\max\bigl\{|x-\al_j|,|x-\be_j|\bigr\}, \quad j\geq1,
\end{equation}
then
\begin{align} \lb{ReflEst2}
\int_\E \f{d\rho(t)}{|t-x|} \leq \f{C_k}{\sqrt{|x-\be_0||x-\al_0|}} \prod_{j\geq1}\f{|x-\ti\ga_j|}{\sqrt{|x-\al_j||x-\be_j|}}.
\end{align}
\end{theorem}
%%%%%%%%%%%%%%%%%%%%%%%%%%%%%%%%%%%%%%%
\begin{proof}
Fix $k\geq1$ and take a point $x\in(\al_k,\be_k)$. Define $\E_\pm=\E\cap(x,\pm\infty)$. % and $m(z)=\int_\E\f{d\rho(t)}{t-z}$, $z\in\bb C\bs\E$.
Since $d\rho$ is absolutely continuous, we have
\begin{equation}
d\rho(t)=\f1\pi \Im[m(t+i0)]dt
\end{equation}
with $m(z)$ as in \eqref{m fct}.
By the reflectionless assumption, $\Im[m(t+i0)]=|m(t+i0)|$  a.e.\;on $\E$,
and hence,
\begin{equation}
d\rho(t)=\f1\pi|m(t+i0)|\chi_\E(t)dt.
\end{equation}
Let $w(t)=|m(t+i0)|$ for a.e.\ $t\in\bb R$, then
\begin{align} \lb{w-prod}
w(t) = \f{1}{\sqrt{|t-\be_0||t-\al_0|}} \prod_{j\geq1}\f{|t-\ga_j|}{\sqrt{|t-\al_j||t-\be_j|}}, \quad t\in\bb R\bs\E.
\end{align}
Define also
\begin{align} \lb{p-prod}
\begin{split}
p_\pm(t) &= \prod_{\substack{j\geq1\\\al_j\gtrless\al_k}} \f{|t-\ga_j|}{\sqrt{|t-\al_j||t-\be_j|}} \quad t\in\bb R\bs\E_\pm,
\\
p_\pm(t) &= \lim_{\eps\downarrow0}\prod_{\substack{j\geq1\\\al_j\gtrless\al_k}} \f{|t+i\eps-\ga_j|}{\sqrt{|t+i\eps-\al_j||t+i\eps-\be_j|}}, \;\text{ a.e. } t\in\E_\pm.
\end{split}
\end{align}
Existence of the limit in \eqref{p-prod} follows from that for $w(t)$ and we have
\begin{align} \lb{wp-prod}
w(t)=\f{p_-(t)}{\sqrt{|t-\be_0||t-\al_k|}} |t-\ga_k| \f{p_+(t)}{\sqrt{|t-\be_k||t-\al_0|}}, \;\text{ a.e. } t\in\bb R.
\end{align}
Define $\ti w(t)$ and $\ti p_{\pm}(t)$ as above, but %in \eqref{w-prod}--\eqref{p-prod}
with $\{\ga_j\}_{j\geq1}$ replaced by $\{\ti\ga_j\}_{j\geq1}$. Then
\begin{align}
p_\pm(t)\leq\ti p_\pm(t)\leq\ti p_\pm(x), \quad t\in[x,\mp\infty).
\end{align}
Since %$d\rho(t)=\f1\pi w(t)\chi_\E(t)dt$ and
$|t-\ga_k|\leq|t-x|+|x-\ti\ga_k|$, we have
\begin{align}
\int_{\E_+} \f{d\rho(t)}{t-x}
&\leq
\f1\pi \int_{\E_+} \f{p_-(t)|x-\ti\ga_k|p_+(t)}{\sqrt{|t-\be_0||t-\al_k|} \sqrt{|t-\be_k||t-\al_0|}} \f{dt}{t-x} \no
\\
&+
\f1\pi \int_{\E_+} \f{p_-(t)|t-x|p_+(t)}{\sqrt{|t-\be_0||t-\al_k|} \sqrt{|t-\be_k||t-\al_0|}} \f{dt}{t-x} \no
\\
&\leq
\f{\ti p_-(x)|x-\ti\ga_k|}{\pi\sqrt{|x-\be_0||x-\al_k|}}
\int_{\E_+} \f{p_+(t)}{\sqrt{|t-\be_k||t-\al_0|}} \f{dt}{t-x} \no
\\
&+
\f{\ti p_-(x)}{\pi} \int_{\E_+} \f{|t-x|p_+(t)}{\sqrt{|t-\be_0||t-\al_k||t-\be_k||t-\al_0|}} \f{dt}{t-x}. \lb{Refl1}
\end{align}
The fact that
\begin{equation}
\f{p_+(t)\chi_{\E_+}(t)dt}{\pi\sqrt{|t-\be_k||t-\al_0|}}
\end{equation}
is the AC part of a reflectionless probability measures on $\E_+$ then gives
\begin{align}
\f1\pi\int_{\E_+} \f{p_+(t)}{\sqrt{|t-\be_k||t-\al_0|}} \f{dt}{t-x} \leq \f{p_+(x)}{\sqrt{|x-\be_k||x-\al_0|}}. \lb{Refl2}
\end{align}
Similarly, noting that
\begin{equation}
\f{|t-x|p_+(t)\chi_{[\be_0,\al_k]\cup\E_+}(t)dt} {\pi\sqrt{|t-\be_0||t-\al_k||t-\be_k||t-\al_0|}}
\end{equation}
is the AC part of a reflectionless probability measure on $[\be_0,\al_k]\cup\E_+$ which is purely AC on $[\be_0,\al_k]$ yields
\begin{align}
\f1\pi\int_{[\be_0,\al_k]\cup\E_+} \f{|t-x|p_+(t)}{\sqrt{|t-\be_0||t-\al_k||t-\be_k||t-\al_0|}} \f{dt}{t-x} \leq 0. \lb{Refl3}
\end{align}
Thus, combining \eqref{Refl2} and \eqref{Refl3} with \eqref{Refl1} gives
\begin{align} \lb{Ellip0}
\int_{\E_+} \f{d\rho(t)}{t-x}
&\leq
\f{\ti p_-(x)}{\sqrt{|x-\be_0||x-\al_k|}}|x-\ti\ga_k| \f{p_+(x)}{\sqrt{|x-\be_k||x-\al_0|}} \no
\\
&-
\f{\ti p_-(x)}{\pi} \int_{\be_0}^{\al_k} \f{|t-x|p_+(t)}{\sqrt{|t-\be_0||t-\al_k||t-\be_k||t-\al_0|}} \f{dt}{t-x} \no
\\
&\leq
\ti w(x) + \f{\ti p_-(x) \ti p_+(x)}{\pi\sqrt{|x-\al_0|}} \int_{\be_0}^{\al_k} \f{dt}{\sqrt{|t-\be_0||t-\al_k||t-\be_k|}}.
\end{align}
We estimate the integral by considering two cases. If $\al_k-\be_0\leq\be_k-\al_k$, then we have $x-\be_0\leq\be_k-\be_0\leq2(\be_k-\al_k)$, and hence,
\begin{align} \lb{Ellip1}
&\int_{\be_0}^{\al_k} \f{dt}{\sqrt{|t-\be_0||t-\al_k||t-\be_k|}} \no
\\
&\quad\leq
\f{1}{\sqrt{|\al_k-\be_k|}}\int_{\be_0}^{\al_k} \f{dt}{\sqrt{|t-\be_0||t-\al_k|}}
\leq \f{\sqrt2\pi}{\sqrt{x-\be_0}}.
\end{align}
Otherwise, $\al_k-\be_0>\be_k-\al_k$ in which case we let $c=(\be_0+\al_k)/2$. Then $x-\be_0\leq\be_k-\be_0\leq2(\al_k-\be_0)=4(c-\be_0)$ and we have
\begin{align} \lb{Ellip2}
&\int_{\be_0}^{\al_k} \f{dt}{\sqrt{|t-\be_0||t-\al_k||t-\be_k|}} \no
\\
&\quad\leq
\f{\int_{\be_0}^c\f{dt}{\sqrt{|t-\be_0|}}}{\sqrt{|c-\al_k||c-\be_k|}}
+
\f{\int_{c}^{\al_k}\f{dt}{\sqrt{|t-\al_k||t-\be_k|}}}{\sqrt{|c-\be_0|}} \no
\\
&\quad= \f{2\sqrt{t-\be_0}\big\vert_{t={\be_0}}^{t=c}} {\sqrt{(\al_k-c)(\be_k-c)}} + \f{-2\log\big(\sqrt{\al_k-t}+\sqrt{\be_k-t}\big)\big\vert_{t=c}^{t=\al_k}} {\sqrt{c-\be_0}}
\no
\\
&\quad\leq \f{2}{\sqrt{\be_k-c}} + \f{2\log\big(\sqrt{2(\be_k-\be_0)}/\sqrt{\be_k-\al_k}\,\big)} {\sqrt{c-\be_0}} \no
\\
&\quad\leq \f{2}{\sqrt{x-\be_0}}\bigg[2 + \log2+\log\f{\be_k-\be_0}{\be_k-\al_k}\bigg].
\end{align}
In the next to last inequality we utilized the Cauchy--Schwarz inequality in the form $\sqrt a+\sqrt b\leq \sqrt{2(a+b)}$. Combining \eqref{Ellip0} with \eqref{Ellip1}--\eqref{Ellip2}, and noting that the estimate in \eqref{Ellip2} is larger than the one in \eqref{Ellip1} and that $2+\log2<3$, then gives
\begin{align}
\int_{\E_+} \f{d\rho(t)}{t-x}
&\leq \ti w(x) + \f{\ti p_-(x) \ti p_+(x)}{\sqrt{|x-\al_0||x-\be_0|}} \bigg[3+\log\f{\be_k-\be_0}{\be_k-\al_k}\bigg].
\end{align}
Since ${|x-\ti\ga_k|}/{\sqrt{|x-\al_k||x-\be_k|}} \geq 1$, we therefore have
\begin{align}
\int_{\E_+} \f{d\rho(t)}{t-x}
&\leq \ti w(x)\bigg[4+\log\f{\be_k-\be_0}{\be_k-\al_k}\bigg]. \lb{Refl4}
\end{align}
In a similar way, one obtains an upper bound for the integral over $\E_-$,
\begin{align}
\int_{\E_-} \f{d\rho(t)}{x-t}
&\leq \ti w(x)\bigg[4+\log\f{\al_0-\al_k}{\be_k-\al_k}\bigg]. \lb{Refl5}
\end{align}
The final step is to note that the integral on the LHS of \eqref{ReflEst1} and \eqref{ReflEst2} can be estimated in two ways, namely
\begin{align}
\int_\E \f{d\rho(t)}{|t-x|} = \bigg|2\int_{\E_\pm} \f{d\rho(t)}{t-x} - \int_\E \f{d\rho(t)}{t-x}\bigg| \leq 2\bigg|\int_{\E_\pm} \f{d\rho(t)}{t-x}\bigg| + \ti w(x).
\end{align}
Combining these estimates with \eqref{Refl4} and \eqref{Refl5}, respectively, and choosing the better bound then yield the result.
\end{proof}
%%%%%%%%%%%%%%%%%%%%%%%%%%%%%%%%%%%%%%%

%%%%%%%%%%%%%%%%%%%%%%%%%%%%%%%%%%%%%%%
\section{Abstract Lieb--Thirring Bounds}
\label{sec3}
%%%%%%%%%%%%%%%%%%%%%%%%%%%%%%%%%%%%%%%

In this section, we obtain Lieb--Thirring bounds for trace class and, more generally, Schatten class perturbations of a wide range of Jacobi matrices.
In particular, our results apply to perturbations of periodic and finite gap Jacobi matrices as well as to several infinite gap Jacobi matrices.
%%%%%%%%%%%%%%%%%%%%%%%%%%%%%%%%%%%%%%%
\begin{theorem} \lb{LTthm1}
Let $J$ and $J'$ be two-sided Jacobi matrices such that $\de J = J-J'$ is in the trace class, that is,
\begin{align}
\sum_{n\in\bb Z} |\de a_n|+|\de b_n|<\infty.
\end{align}
Let $\E=\si(J')$ and denote the gaps of $\E$ as in \eqref{SetE}. In addition, suppose there exist non-negative constants $\{C_k\}_{k\geq0}$ such that for some $1/2<p<1$,
\begin{align} \lb{CkCond1}
\sum_{k\geq1} C_k(\be_k-\al_k)^{p-1/2} < \infty
\end{align}
and such that the spectral measures $d\rho_n$ of $(J',\de_n)$ satisfy
\begin{align} \lb{RnCond1}
\sup_{n\in\bb Z}\int_\E \f{d\rho_n(t)}{|t-x|} \leq
\begin{cases}
\displaystyle{
\f{C_0}{|x-\al_0|^{1/2}|x-\be_0|^{1/2}}, \quad x\in(\al_0,\be_0),} \vspace{0.2cm}\\
\displaystyle{
\qquad \f{C_k}{\dist(x,\E)^{1/2}}, \quad x\in(\al_k,\be_k), \quad k\geq1}.
\end{cases}
\end{align}
Then $\si_\ess(J)=\E$ and the discrete eigenvalues of $J$ satisfy the Lieb--Thirring bound,
\begin{align} \lb{LTineq1}
\sum_{\la\in\si(J)\bs\E} \dist(\la,\E)^p \leq L_{p,\,J'} \sum_{n\in\bb Z} 4|\de a_n|+|\de b_n|,
\end{align}
where the constant $L_{p,\,J'}$ is independent of $\de J$ and explicitly given by
\begin{align}
L_{p,\,J'} =\f{p}{2p-1}\bigg( \f{C_0}{(1-p)(\al_0-\be_0)^{1-p}} + 2 \sum_{k\geq1} C_k\Big(\f{\be_k-\al_k}{2}\Big)^{p-1/2} \bigg).
\end{align}
\end{theorem}
%%%%%%%%%%%%%%%%%%%%%%%%%%%%%%%%%%%%%%%
\begin{proof}
Assumption \eqref{RnCond1} implies that the spectral measures $d\rho_n$ of $J'$ cannot have point masses at the endpoints of $\E$ (i.e., $\{\al_k,\be_k\}_{k\geq0}$).
Thus, $J'$ has no isolated eigenvalues, and hence, $\si_\ess(J')=\si(J')=\E$. Weyl's theorem then yields $\si_\ess(J)=\E$ since $J$ is a compact perturbation of $J'$.

Let $(c)_\pm = \max(\pm c,0)$ and define tridiagonal matrices $\de J_\pm$ and diagonal matrices $D_\pm$ by
\begin{align}
&(\de J_\pm)_{n,n-1} = \pm\mfrac{1}{2} \de a_{n-1}, \;
(\de J_\pm)_{n,n+1} = \pm\mfrac{1}{2} \de a_n, \no
\\
&(\de J_\pm)_{n,n} = (\de b_n)_\pm + \mfrac{1}{2}|\de a_n| + \mfrac{1}{2}|\de a_{n-1}|, \lb{dJdef}
\\[1.5mm]
&(D_\pm)_{n,n} = (\de b_n)_\pm + |\de a_n| + |\de a_{n-1}|, \quad n\in\bb Z. \lb{Ddef}
\end{align}
Then $\de J = \de J_+ - \de J_-$ and $0 \leq \de J_\pm \leq D_\pm$ since
\begin{gather}
\begin{pmatrix}0 & \de a_n\\ \de a_n & 0\end{pmatrix} =
\f12\begin{pmatrix}|\de a_n| & \de a_n\\ \de a_n & |\de a_n|\end{pmatrix} -
\f12\begin{pmatrix}|\de a_n| & -\de a_n\\ -\de a_n & |\de a_n|\end{pmatrix},
\\[1mm]
0 \leq
\f12\begin{pmatrix}|\de a_n| & \pm \de a_n\\ \pm\de a_n & |\de a_n|\end{pmatrix}
\leq \begin{pmatrix}|\de a_n| & 0\\0 & |\de a_n|\end{pmatrix}.
\end{gather}

Let $N(J\in I)$ denote the number of eigenvalues of $J$ contained in an interval $I\subset\bb R\bsn\E$.
Then by a version of the Birman--Schwinger principle \cite[Theorem 1.4]{FS11}, for a.e.\ $\ga_\pm$ such that $[\ga_-,\ga_+]\subset\bb R\bsn\E$,
\begin{align} \lb{BrmSchw1}
& N\big(J \in  (\ga_-,\ga_+)\big) =
N\big(J' + \de J_+ - \de J_- \in (\ga_-,\ga_+)\big) \no
\\
&\; \leq N\big(\de J_+^{1/2}(J'-\ga_-)^{-1}\de J_+^{1/2}<-1\big)
+ N\big(\de J_-^{1/2}(J'-\ga_+)^{-1}\de J_-^{1/2}>1\big) \no
\\
&\; \leq \|\de J_+^{1/2}(J'-\ga_-)^{-1}\de J_+^{1/2}\|_1 + \|\de J_-^{1/2}(J'-\ga_+)^{-1}\de J_-^{1/2}\|_1 \no
\\
&\; \leq \|D_+^{1/2}(J'-\ga_-)^{-1}D_+^{1/2}\|_1 + \|D_-^{1/2}(J'-\ga_+)^{-1}D_-^{1/2}\|_1,
\end{align}
where the last inequality follows from the fact that $D_\pm\geq\de J_\pm\geq0$.
By assumption \eqref{RnCond1} and Theorem~\ref{TNEthm}, we get that
\begin{align} \lb{TrNormEst1}
\|D_\pm^{1/2}(J'-x)^{-1}D_\pm^{1/2}\|_1 \leq
\begin{cases}
\displaystyle{
\f{C_0\|D_\pm\|_1}{|x-\al_0|^{1/2}|x-\be_0|^{1/2}}, \quad x\in(\al_0,\be_0),} \vspace{0.2cm} \\
\displaystyle{
\qquad \f{C_k\|D_\pm\|_1}{\dist(x,\E)^{1/2}}, \quad x\in(\al_k,\be_k), \quad k\geq1.}
\end{cases}
\end{align}
Let $\ell_0=\infty$, $\ell_k=(\be_k-\al_k)/2$ for $k\geq1$, and set $d=|\al_0-\be_0|$. Then writing the LHS of \eqref{LTineq1} as
\begin{align}
\sum_{\la\in\si(J)\bs\E} \dist(\la,\E)^p = \sum_{k\geq0} \int_0^{\ell_k} (x^p)'N\big(J\in(\al_k+x,\be_k-x)\big)dx,
\end{align}
we can estimate using \eqref{BrmSchw1} and \eqref{TrNormEst1} to get
\begin{align}
&\sum_{\la\in\si(J)\bs\E} \dist(\la,\E)^p \leq \big( \|D_+\|_1+\|D_-\|_1 \big) \no
\\
&\qquad\times\bigg(
C_0 \int_0^{\infty} \f{px^{p-1}}{x^{1/2}(x+d)^{1/2}}dx +
\sum_{k\geq1} C_k\int_0^{\ell_k} \f{px^{p-1}}{x^{1/2}}dx \bigg).
\end{align}
As the first integral is bounded by
\begin{align}
\int_0^d \f{px^{p-1}}{x^{1/2}d^{1/2}}dx + \int_d^\infty\f{px^{p-1}}{x^{1/2}x^{1/2}}dx,
\end{align}
we have
\begin{align}
&\sum_{\la\in\si(J)\bs\E} \dist(\la,\E)^p \leq \big( \|D_+\|_1+\|D_-\|_1 \big) \no
\\
&\qquad \times \bigg(\f{p}{p-1/2} C_0 d^{p-1} + \f{p}{1-p} C_0 d^{p-1} +
\f{p}{p-1/2} \sum_{k\geq1} C_k \ell_k^{p-1/2} \bigg).
\end{align}
Combining this with \eqref{Ddef} then yields \eqref{LTineq1}.
\end{proof}
%%%%%%%%%%%%%%%%%%%%%%%%%%%%%%%%%%%%%%%

In the next theorem, we extend our Lieb--Thirring bounds to non trace class perturbations.
%%%%%%%%%%%%%%%%%%%%%%%%%%%%%%%%%%%%%%%
\begin{theorem} \lb{LTthm2}
Let $J$ and $J'$ be two-sided Jacobi matrices such that $\de J = J-J'$ is in the Schatten class $\cc S_p$ for some $p>1$, that is, $($cf.\ \cite[Lemma 2.3]{KS03}$)$
\begin{align}
\sum_{n\in\bb Z} |\de a_n|^p+|\de b_n|^p < \infty.
\end{align}
Let $\E=\si(J')$ and denote the gaps of $\E$ as in \eqref{SetE}. In addition, suppose there exist non-negative constants $\{C_k\}_{k\geq0}$ such that
\begin{align} \lb{CkCond2}
\sum_{k\geq0} C_k < \infty
\end{align}
and such that the spectral measures $d\rho_n$ of $(J',\de_n)$ satisfy
\begin{align} \lb{RnCond2}
\sup_{n\in\bb Z}\int_\E \f{d\rho_n(t)}{|t-x|} \leq \f{C_k}{\dist(x,\E)^{1/2}}, \quad x\in(\al_k,\be_k), \quad k\geq0.
\end{align}
Then $\si_\ess(J)=\E$ and the discrete eigenvalues of $J$ satisfy the Lieb--Thirring bound
\begin{align} \lb{LTineq2}
\sum_{\la\in\si(J)\bs\E} \dist(\la,\E)^{p-1/2} \leq L_{p,\,J'} \sum_{n\in\bb Z} 4|\de a_n|^p+|\de b_n|^p,
\end{align}
where the constant $L_{p,\,J'}$ is independent of $\de J$ and explicitly given by
\begin{align}
L_{p,\,J'} = 2^{p-3/2}3^{p-1}\f{2p-1}{p-1} \sum_{k\geq0} C_k.
\end{align}
\end{theorem}
%%%%%%%%%%%%%%%%%%%%%%%%%%%%%%%%%%%%%%%
\begin{proof}
As in the previous theorem, assumption \eqref{RnCond2} implies that $J'$ has no isolated eigenvalues. %, that is, $\si_\ess(J')=\E$.
Since $J$ is a compact perturbation of $J'$, it follows that $\si_\ess(J)=\E$.

Define compact operators $\de J_\pm$ and $D_\pm$ as in \eqref{dJdef}--\eqref{Ddef}.
Then $\de J=\de J_+ - \de J_-$ and $0\leq\de J_\pm \leq D_\pm$.
Let $N(J\in I)$ denote the number of eigenvalues of $J$ contained in an interval $I\subset\bb R\bsn\E$.
For $\la\in\bb R\bsn\E$, we denote by $N^\pm_\la(J',\de J_\pm)$ the number of eigenvalues of $J'\pm x\de J_\pm$ that pass through $\la$ as $x$ runs
through the interval $(0,1)$. By a version of the Birman--Schwinger principle \cite[Theorem 1.4]{FS11}), for a.e.\ $\ga_\pm$ such that $[\ga_-,\ga_+]\subset\bb R\bsn\E$,
\begin{align} \lb{BrmSchw2.1}
& N\big(J\in(\ga_-,\ga_+)\big) \leq N^+_{\ga_-}(J',\de J_+) + N^-_{\ga_+}(J',\de J_-),
\\[1mm]
&\ N^\pm_{\la}(J',\de J_\pm) = N\big(\de J_\pm^{1/2}(J'-\la)^{-1}\de J_\pm^{1/2}\lessgtr \mp 1\big).
\end{align}
Since $D_\pm\geq\de J_\pm\geq0$, we have $N^\pm_{\la}(J',\de J_\pm) \leq N^\pm_{\la}(J',D_\pm)$, and hence,
\begin{align} \lb{BrmSchw2.2}
N\big(J\in(\ga_-,\ga_+)\big) \leq N^+_{\ga_-}(J',D_+) + N^-_{\ga_+}(J',D_-).
\end{align}
To handle non trace class perturbations, we estimate further in terms of finite rank truncated versions of $D_\pm$.
For this, let $0<r<\dist(\la,\E)$ and define the finite rank diagonal matrices $D_{\pm,r}$ by
\begin{equation}
(D_{\pm,r})_{n,n} = ((D_\pm)_{n,n}-r)_+.
\end{equation}
Then $\|D_\pm-D_{\pm,r}\| \leq r$ so the eigenvalues of $J'+D_{\pm,r}+x(D_\pm-D_{\pm,r})$ can move a distance of no more than $r$
as $x$ ranges from $0$ to $1$. Thus,
\begin{align} \lb{SpEst2.2}
N^\pm_\la(J',D_\pm) \leq N^\pm_{\la\mp r}(J',D_{\pm,r}) = N\big(D_{\pm,r}^{1/2}(J'-\la\pm r)^{-1}D_{\pm,r}^{1/2}\lessgtr \mp 1\big).
\end{align}
Estimating the RHS by the trace norm, applying Theorem~\ref{TNEthm}, and using the assumption \eqref{RnCond2} then yield
\begin{align} \lb{TrNormEst2}
N^\pm_\la(J',D_\pm) \leq
\|D_{\pm,r}^{1/2}(J'-\la\pm r)^{-1}D_{\pm,r}^{1/2}\|_1
\leq \f{C_k\|D_{\pm,r}\|_1}{\dist(\la\mp r,\E)^{1/2}}
\end{align}
whenever $\la\mp r\in(\al_k,\be_k)$, $k\geq0$.

Let $\ell_0=\infty$, $\ell_k=(\be_k-\al_k)/2$ for $k\geq1$, and $d^\pm_n=(D_\pm)_{n,n}$ for $n\in\bb Z$.
Applying \eqref{BrmSchw2.2} to an interval $[\al_k+x,\be_k-x]$ and using \eqref{TrNormEst2} with $r=x/2$ then gives for a.e.\ $x\in(0,\ell_k)$, $k\geq0$,
\begin{align} \lb{Est2}
N\big(J\in(\al_k+x,\be_k-x)\big) &\leq
N^+_{\al_k+x}(J',D_+) + N^-_{\be_k-x}(J',D_-) \no
\\
&\leq
C_k\big(\|D_{+,\f x2}\|_1+\|D_{-,\f x2}\|_1\big)(x/2)^{-1/2}
\\
&\leq
C_k\sum_{n\in\bb Z}\big((2d^+_n-x)_++(2d^-_n-x)_+\big)(2x)^{-1/2}. \no
\end{align}
Write the LHS of \eqref{LTineq2} as an integral and estimate by use of \eqref{Est2} to get
\begin{align}
&\sum_{\la\in\si(J)\bs\E} \dist(\la,\E)^{p-1/2} = \sum_{k\geq0} \int_0^{\ell_k} (x^{p-1/2})'N\big(J\in(\al_k+x,\be_k-x)\big)dx \no
\\
&\quad \leq
\f{2p-1}{2^{3/2}} \sum_{k\geq0} C_k \int_0^{\ell_k} \sum_{n\in\bb Z}  \big((2d^+_n-x)_++(2d^-_n-x)_+\big)x^{p-2}dx.
\end{align}
Rearranging the integral and the sum over $n$ by the monotone convergence theorem and estimating the integrals by
\begin{align}
\int_0^{\ell_k}(2d^\pm_n-x)_+x^{p-2}dx \leq \int_0^{2d^\pm_n}2d^\pm_n x^{p-2}dx \leq \f{(2d^\pm_n)^p}{p-1},
\end{align}
give
\begin{align}
&\sum_{\la\in\si(J)\bs\E} \dist(\la,\E)^{p-1/2}
\leq
2^{p-3/2}\f{2p-1}{p-1}\sum_{k\geq0} C_k\sum_{n\in\bb Z} (d^+_n)^p+(d^-_n)^p.
\end{align}
Recalling \eqref{Ddef} and using Jensen's convexity inequality lead to \eqref{LTineq2}.
\end{proof}
%%%%%%%%%%%%%%%%%%%%%%%%%%%%%%%%%%%%%%%

Several remarks pertaining to the previous two theorems are in order.
%%%%%%%%%%%%%%%%%%%%%%%%%%%%%%%%%%%%%%%
\begin{remark} \lb{LTremark}
(a) The Jacobi matrix $J'$ is not required to be reflectionless, that is, $J'$ is not necessarily from the isospectral torus $\cc T_\E$.
The only restrictions on $J'$ are the conditions \eqref{CkCond1}--\eqref{RnCond1} in Theorem~\ref{LTthm1} and \eqref{CkCond2}--\eqref{RnCond2} in Theorem~\ref{LTthm2}, respectively.

(b) If $E$ is a finite gap set and $J'\in\cc T_\E$, then the assumptions \eqref{CkCond1}--\eqref{RnCond1} and \eqref{CkCond2}--\eqref{RnCond2}
are trivially satisfied. In this case, the first theorem extends a result of \cite{HS08} by providing an explicit constant for the RHS of \eqref{LTineq1}
and the second theorem complements a recent result of \cite{FS11} for $p=1/2$.

(c) If $E$ is a homogeneous set and $J'\in\cc T_\E$, then the spectral measures $d\rho_n$ of $J'$ are absolutely continuous (cf., e.g., \cite{PR09,PSZ10}),
and hence, by Theorem \ref{ReflEstThm} it is possible to replace
\begin{equation}
\sup_{n\in\bb Z}\int_\E\f{d\rho_n(t)}{|t-x|} \quad \mbox{ by} \quad \sup_{d\mu\in\cc R_\E}\bigg|\int_\E\f{d\mu(t)}{t-x}\bigg|
\end{equation}
while simultaneously changing
\begin{equation}
C_k \quad \mbox{to} \quad C_k/\log\f{\al_0-\be_0}{\be_k-\al_k}, \quad k\geq1,
\end{equation}
in \eqref{RnCond1} and \eqref{RnCond2}, respectively. In this case, the constants $L_{p,\,J'}$ in \eqref{LTineq1} and \eqref{LTineq2} are replaced
by a constant $L_{p,\,\E}$ which is uniform in $J'\in\cc T_\E$ and only depends on $p$ and $\E$.

(d) Theorems \ref{LTthm1}--\ref{LTthm2} also extend to perturbations of Jacobi matrices $J'$ that exhibit a different behavior near the gaps edges.
%in conditions \eqref{RnCond1} and \eqref{RnCond2}.
For example, if $J'$ satisfies \eqref{RnCond1} and \eqref{RnCond2} with power $1/2$ replaced by $1/2+q$ for some $q\geq0$, then the Lieb--Thirring bounds continue
to hold with $p$ replaced by $p+q$ on the LHS of \eqref{LTineq1} and \eqref{LTineq2} and appropriately adjusted constants $L_{p,\,J'}$.

(e) By the Aronszajn--Donoghue theory of rank one perturbations (see, for example, \cite[Sect.~12.2]{Si05}), $\la\in\bb R\bsn\E$ is an eigenvalue of a rank one
perturbation $J=J'+\de b_n\langle\de_n,\,\cdot\,\rangle\de_n$ if and only if
\begin{align}
\label{rank one}
\int_\E\f{d\rho_n(t)}{t-\la}=\bigl\langle\de_n,(J'-\la)^{-1}\de_n\bigr\rangle = -\f{1}{\de b_n}.
\end{align}
Thus, a necessary condition for the following Lieb--Thirring bound
\begin{align} \lb{pqLTineq}
\sum_{\la\in\si(J)\bs\E} \dist(\la,\E)^{p} \leq  L_{p,\,q}\sum_{n\in\bb Z} |\de a_n|^q+|\de b_n|^q, \quad q>p>0,
\end{align}
to hold is
\begin{align}
\bigg|\int_\E\f{d\rho_n(t)}{t-\la}\bigg| = \f{1}{|\de b_n|} \leq \f{L_{p,\,q}^{1/q}}{\dist(\la,\E)^{p/q}}.
\end{align}
Moreover, since
\begin{equation}
\int_\E\f{d\rho_n(t)}{|t-x|}-\bigg|\int_\E\f{d\rho_n(t)}{t-x}\bigg|
\end{equation}
is bounded in each gap, the conditions
\begin{align}
\sup_{n\in\bb Z}\int_\E \f{d\rho_n(t)}{|t-x|} \leq \f{C_k}{\dist(x,\E)^{p/q}}, \quad x\in(\al_k,\be_k), \quad k\geq0,
\end{align}
for some constants $C_k>0$ are necessary for \eqref{pqLTineq} to hold. Thus, the assumptions \eqref{RnCond1} and \eqref{RnCond2} in our theorems are close to being necessary.
\end{remark}
%%%%%%%%%%%%%%%%%%%%%%%%%%%%%%%%%%%%%%%

%%%%%%%%%%%%%%%%%%%%%%%%%%%%%%%%%%%%%%%
\section{Examples}
\label{sec4}
%%%%%%%%%%%%%%%%%%%%%%%%%%%%%%%%%%%%%%%

In this section, we obtain Lieb--Thirring bounds for perturbations of Jacobi matrices from the isospectral tori, $\cc T_\E$, for two explicit classes of homogeneous infinite gap sets. The isospectral torus associated with a homogeneous set $\E$ is known to consist of almost periodic Jacobi matrices, see \cite{JSC2, SY97}. We also recall that reflectionless measures on homogeneous sets are necessarily absolutely continuous \cite{PR09,PSZ10}.

%%%%%%%%%%%%%%%%%%%%%%%%%%%%%%%%%%%%%%%
\subsection{Infinite Band Example}
\label{sec4a}
%%%%%%%%%%%%%%%%%%%%%%%%%%%%%%%%%%%%%%%

In this subsection, we consider an explicit example of a compact set $\E$ which consists of infinitely many disjoint intervals that accumulate
at $\inf\E$. Suppose $\{\eps_k\}_{k=1}^\infty\subset(0,1)$ and let
\begin{equation} \lb{InfBandSetE}
\E=\bigcap_{k=0}^\infty\E_k,
\end{equation}
where $\E_0=[\be_0,\al_0]$ and $\E_{k}$ is the compact set obtained from $\E_{k-1}$ by removing the middle $\eps_{k}$ portion from the first
of the $k$ bands in $\E_{k-1}$. We will denote the gap at level $k$ by $(\al_k,\be_k)$, that is,
\begin{equation}
(\al_k,\be_k)=\E_{k-1}\bsn\E_{k}, \quad k\geq1.
\end{equation}
It is easy to see that $\E$ is a homogeneous set if and only if $\sup_{k\geq1}\eps_k<1$.
%%%%%%%%%%%%%%%%%%%%%%%%%%%%%%%%%%%%%%%
\begin{theorem} \lb{InfBandE}
Suppose $\E$ is the infinite band set constructed in \eqref{InfBandSetE}. If $\sum_{k=1}^\infty \eps_k<\infty$, then for some constant $C>0$,
\begin{align} \lb{InfBandEst}
\sup_{d\rho\in\cc R_\E}\bigg|\int_\E\f{d\rho(t)}{t-x}\bigg| \leq
\begin{cases}
\ds\f{C}{|x-\al_0|^{1/2}|x-\be_0|^{1/2}}, \quad x\in(\al_0,\be_0), \\[5mm]
\ds\qquad\f{C\sqrt{\eps_k}}{\dist(x,\E)^{1/2}}, \quad x\in(\al_k,\be_k), \quad k\geq1.
\end{cases}
\end{align}
Conversely, if
\begin{align} \lb{InfBandEst2}
\limsup_{x\nearrow\be_0}\, |x-\be_0|^{1/2}\!\sup_{d\rho\in\cc R_\E}\bigg|\int_\E\f{d\rho(t)}{t-x}\bigg| < \infty,
\end{align}
then $\sum_{k=1}^\infty \eps_k<\infty$.
\end{theorem}
%%%%%%%%%%%%%%%%%%%%%%%%%%%%%%%%%%%%%%%
\begin{proof}
First assume $\sum_{k=1}^\infty\eps_k<\infty$ and let $d\rho$ be a reflectionless probability measure on $\E$. Fix $k\geq1$ and define
\begin{align} \lb{InfBand-pp}
& p_+(x) = \prod_{j=1}^{k-1}\f{|x-\ga_j|}{\sqrt{|x-\al_j||x-\be_j|}}, \quad
p_-(x) = \prod_{j=k+1}^\infty\f{|x-\ga_j|}{\sqrt{|x-\al_j||x-\be_j|}},
\end{align}
where $\ga_j\in[\al_j,\be_j]$, $j\geq1$, are chosen in such a way that
\begin{align} \lb{InfBand-rp}
&d\rho(t) = \f{p_-(t)|t-\ga_k|p_+(t)\chi_\E(t)dt}{\pi\sqrt{|t-\be_0||t-\al_k||t-\be_k||t-\al_0|}}.
\end{align}
Equivalently,
\begin{align} \lb{InfBand-wp}
&\bigg|\int_\E\f{d\rho(t)}{t-x}\bigg| = \f{p_-(x)|x-\ga_k|p_+(x)}{\sqrt{|x-\be_0||x-\al_k||x-\be_k||x-\al_0|}}, \quad x\in\bb R\bsn\E.
\end{align}
In addition, let $b_0=\al_0-\be_0$ and
\begin{equation}
b_j = \al_j-\be_0 = \al_{j-1}-\be_j, \quad g_j=\be_j-\al_j, \quad j\geq1,
\end{equation}
be the band and gap lengths at level $j$. Then it follows from the construction of $\E_j$ that
\begin{align}
& b_j = \f{1-\eps_j}{2}b_{j-1}, \quad g_j = \eps_j b_{j-1} = \f{2\eps_j}{1-\eps_j}b_j, \quad j\geq1.
\end{align}
Letting $c=\min_{j\geq1}(1-\eps_j)(1-\eps_{j+1})$, we can estimate $p_\pm(x)$ as follows
\begin{align} \lb{InfBand-up1}
p_+(x) &\leq \prod_{j=1}^{k-1}\sqrt{\f{|x-\be_j|}{|x-\al_j|}}
\leq
\prod_{j=1}^{k-1}\sqrt{\f{\be_j-\be_k}{\al_j-\be_k}}
\leq
\exp\bigg\{\f12\sum_{j=1}^{k-1}\f{\be_j-\al_j}{\al_j-\be_k}\bigg\}\no
\\
&\leq
\exp\bigg\{\f12\sum_{j=1}^{k-1}\f{g_j}{b_{j+1}}\bigg\}
\leq
\exp\bigg\{\f2c\sum_{j=1}^{k-1}\eps_j\bigg\},
\quad
x \leq \be_k,
\end{align}
and similarly,
\begin{align} \lb{InfBand-up2}
p_-(x) &\leq \prod_{j=k+1}^{\infty}\sqrt{\f{|x-\al_j|}{|x-\be_j|}}
\leq
\prod_{j=k+1}^{\infty}\sqrt{\f{\al_k-\al_j}{\al_k-\be_j}}
\leq \exp\bigg\{\f12\sum_{j=k+1}^\infty\f{\be_j-\al_j}{\al_k-\be_j}\bigg\} \no
\\
&\leq
\exp\bigg\{\f12\sum_{j=k+1}^\infty\f{g_j}{b_j}\bigg\}
\leq
\exp\bigg\{\f1c\sum_{j=k+1}^{\infty}\eps_j\bigg\}, \quad
x\geq \al_k.
\end{align}

Now suppose $x\in(\al_k,\be_k)$. Then since $\ga_k\in[\al_k,\be_k]$ and
\begin{equation}
\be_k-\al_k=\f{2\eps_k}{1-\eps_k}(\al_k-\be_0),
\end{equation}
the estimates \eqref{InfBand-up1}--\eqref{InfBand-up2} combined with \eqref{InfBand-wp} yield
\begin{align}
\bigg|\int_\E\f{d\rho(t)}{t-x}\bigg| &\leq \f{\exp\big\{\f2c\sum_{j=1}^{\infty}\eps_j\big\}} {\sqrt{|x-\be_0||x-\al_0|}} \f{|x-\ga_k|}{\sqrt{|x-\al_k||x-\be_k|}} \no
\\
&\leq \f{\exp\big\{\f2c\sum_{j=1}^{\infty}\eps_j\big\}} {\sqrt{|\al_k-\be_0||\be_k-\al_0|}} \sqrt{\f{\be_k-\al_k}{\dist(x,\E)}}
\leq \f{C\sqrt{\eps_k}}{\dist(x,\E)^{1/2}},
\end{align}
where $C$ is a constant that depends only on $\E$. This proves the second and more involved part of \eqref{InfBandEst}.

To handle the case of $x\in(\al_0,\be_0)$, let $p_+(x)$ and $p_-(x)$ be defined as in \eqref{InfBand-pp} but with $k=\infty$ and $k=0$, respectively. Then
\begin{align} \lb{InfBand-wp2}
\bigg|\int_\E\f{d\rho(t)}{t-x}\bigg|=\f{p_+(x)}{\sqrt{|x-\be_0||x-\al_0|}} = \f{p_-(x)}{\sqrt{|x-\be_0||x-\al_0|}}
\end{align}
and just as for the above estimates, we get $p_+(x)\leq \exp\big\{\f2c\sum_{j=1}^\infty\eps_j\big\}$ for $x\leq \be_0$
and $p_-(x)\leq \exp\big\{\f1c\sum_{j=1}^\infty\eps_j\big\}$ for $x\geq\al_0$. Thus, \eqref{InfBandEst} follows.

For the converse direction, assume that \eqref{InfBandEst2} holds. Let $d\rho$ be the reflectionless measure on $\E$ that corresponds to $\ga_j=\be_j$
for every $j\geq1$ and let $p_+(t)$ be defined as in \eqref{InfBand-pp} with $k=\infty$. Then $p_+(x)\to p_+(\be_0)$ as $x\nearrow\be_0$
and since $1+x\geq\exp(x/2)$ for $x\in[0,2]$, we have
\begin{align} \lb{InfBand-low}
p_+(\be_0) &= \prod_{j=1}^{\infty}\sqrt{\f{\be_j-\be_0}{\al_j-\be_0}}
= \prod_{j=1}^{\infty}\sqrt{1+\f{g_j}{b_j}} \no
\\
&\geq \prod_{j=1}^{\infty}\sqrt{1+2\eps_j}
\geq \exp\bigg\{\f12\sum_{j=1}^{\infty}\eps_j\bigg\}.
\end{align}
%Strictly speaking, there may be a finite number of $j$'s for which $g_j/b_j>2$. In any event,
Thus, $\sum_{j=1}^\infty\eps_j<\infty$ follows from \eqref{InfBand-low}, \eqref{InfBand-wp2}, and \eqref{InfBandEst2}.
\end{proof}
%%%%%%%%%%%%%%%%%%%%%%%%%%%%%%%%%%%%%%%

%%%%%%%%%%%%%%%%%%%%%%%%%%%%%%%%%%%%%%%
Our abstract results in Theorems~\ref{LTthm1} and \ref{LTthm2} combined with the estimate derived in Theorems~\ref{InfBandE} and \ref{ReflEstThm} yield the
following Lieb--Thirring bounds.

%%%%%%%%%%%%%%%%%%%%%%%%%%%%%%%%%%%%%%%
\begin{theorem} \lb{InfBandLT}
%Suppose $\{\eps_k\}_{k=1}^\infty\subset(0,1)$ and
Let $\E$ be the infinite band set constructed in \eqref{InfBandSetE}
and suppose $J$, $J'$ are two-sided Jacobi matrices such that $J'\in\cc T_\E$ and $J=J'+\de J$ is a compact perturbation of $J'$.
If $\sum_{k=1}^\infty \eps_k<\infty$, then
\begin{align} \lb{InfBand-LT1}
\sum_{\la\in\si(J)\bs\E} \dist(\la,\E)^p \leq L_{p,\,\E} \sum_{n\in\bb Z} |\de a_n|+|\de b_n|
\end{align}
for $1/2<p<1$.
If, in addition, $\sum_{k=1}^\infty \sqrt{\eps_k}\log(1/\eps_k) <\infty$, then
\begin{align} \lb{InfBand-LT2}
\sum_{\la\in\si(J)\bs\E} \dist(\la,\E)^{p} \leq  L_{p,\,\E}\sum_{n\in\bb Z} |\de a_n|^{p+1/2}+|\de b_n|^{p+1/2}
\end{align}
for every $p>1/2$.
In either case, the constant $L_{p,\,\E}$ is independent of $J$ and $J'$ and only depends on $p$ and $\E$.
\end{theorem}
%%%%%%%%%%%%%%%%%%%%%%%%%%%%%%%%%%%%%%%
\begin{proof}
Recall that every reflectionless measure on $\E$ is absolutely continuous since $\E$ is a homogeneous set.
By construction of $\E$,
\begin{align}
\f{\be_k-\be_0}{\be_k-\al_k} = 1+\f{\al_k-\be_0}{\be_k-\al_k} = 1+\f{1-\eps_k}{2\eps_k} \leq \f{1}{\eps_k}, \quad k\geq 1.
\end{align}
Thus, \eqref{InfBandEst} combined with \eqref{ReflEst1} yields \eqref{RnCond1} and \eqref{RnCond2} for the gap at level $k\geq1$ with a constant
\begin{equation} \lb{InfBandCk}
C_k=C\sqrt{\eps_k}\log(1/\eps_k),
\end{equation}
where $C>0$ is sufficiently large and independent of $k$.
Since
\begin{equation}
\be_k-\al_k \leq 2^{1-k}\eps_k(\al_0-\be_0), \quad k\geq1,
\end{equation}
\eqref{CkCond1} is satisfies due to the exponential decay of $(\be_k-\al_k)^{p-1/2}$. Moreover, \eqref{CkCond2} holds by assumption.
Thus, \eqref{InfBand-LT1} and \eqref{InfBand-LT2} follow from Theorems \ref{LTthm1} and \ref{LTthm2}, respectively.
\end{proof}
%%%%%%%%%%%%%%%%%%%%%%%%%%%%%%%%%%%%%%%

%%%%%%%%%%%%%%%%%%%%%%%%%%%%%%%%%%%%%%%
In addition to Theorem \ref{InfBandLT}, we have the following result in which the distance to the essential spectrum is measured by the potential
theoretic Green function $g$ instead of the usual distance function. The proof relies on the well-known relation between the Green function and
the equilibrium measure for $\E$, denoted $d\mu_\E$,
\begin{equation} \lb{g-dmuE}
g(z)=\ga(\E)-\int \log{|z-t|}^{-1} d\mu_\E(t), \quad z\in\bb C\bsn\E,
\end{equation}
where $\ga(\E)=-\log\bigl(\ca(\E)\bigr)$ is the so-called Robin constant for $\E$.

%%%%%%%%%%%%%%%%%%%%%%%%%%%%%%%%%%%%%%%
\begin{theorem} \lb{InfBandGLT}
Let $\E$ be the infinite band set constructed in \eqref{InfBandSetE} and suppose
$J$, $J'$ are two-sided Jacobi matrices such that $J'\in\cc T_\E$ and $J=J'+\de J$ is a compact perturbation of $J'$.
If $\sum_{k=1}^\infty \eps_k<\infty$, then for every $p>1$,
\begin{align} \lb{InfBand-LT3}
\sum_{\la\in\si(J)\bs\E} g(\la)^{p} \leq  L_{p,\,\E}\sum_{n\in\bb Z} |\de a_n|^{(p+1)/2}+|\de b_n|^{(p+1)/2},
\end{align}
where the constant $L_{p,\,\E}$ is independent of $J$, $J'$ and only depends on $p$ and $\E$.
\end{theorem}
%%%%%%%%%%%%%%%%%%%%%%%%%%%%%%%%%%%%%%%
\begin{proof}
Let $\pd=\f12(\f{\pd}{\pd x}-i\f{\pd}{\pd y})$, then for any analytic function $f(z)$ we have $2\pd\big(\Re f(z)\big)=f'(z)$ by the Cauchy--Riemann equations. Combining this observation with \eqref{g-dmuE} yields
\begin{equation} \lb{dg-dmuE}
2\pd g(z)=\int_\E \f{d\mu_\E(t)}{z-t}, \quad z\in\bb C\bsn\E.
\end{equation}

For convenience, we define $\eps_0=1/e$. Then since the equilibrium measure $d\mu_\E$ is reflectionless on $\E$, it follows from \eqref{InfBandEst} that
\begin{align}
|\pd g(x)| \leq \f{C\sqrt{\eps_k}}{\dist(x,\E)^{1/2}}, \quad x\in(\al_k, \be_k), \quad k\geq0.
\end{align}
Recalling that the Green function vanishes on $\E$, integration over the gaps then gives
\begin{align}
g(x) \leq C\sqrt{\eps_k}\,\dist(x,\E)^{1/2}, \quad x\in(\al_k, \be_k), \quad k\geq0.
\end{align}
As in the proofs of Theorems \ref{LTthm2} and \ref{InfBandLT}, we hence get
\begin{align}
\sum_{\la\in\si(J)\cap(\al_k,\be_k)} \dist(\la,\E)^{p/2} \leq  C_k\sum_{n\in\bb Z} |\de a_n|^{(p+1)/2}+|\de b_n|^{(p+1)/2}, %\quad k\geq0,
\end{align}
where $C_k=C\sqrt{\eps_k}\log(1/\eps_k)$. Thus, for each $k\geq0$,
\begin{align}
\sum_{\la\in\si(J)\cap(\al_k,\be_k)} g(\la)^p \leq  C\eps_k^{(p+1)/2}\log(1/\eps_k)\sum_{n\in\bb Z} |\de a_n|^{(p+1)/2}+|\de b_n|^{(p+1)/2},
\end{align}
and since $\eps_k^{(p-1)/2}\log(1/\eps_k)$ is a bounded sequence, summing over $k$ yields \eqref{InfBand-LT3}.
\end{proof}
%%%%%%%%%%%%%%%%%%%%%%%%%%%%%%%%%%%%%%%

%%%%%%%%%%%%%%%%%%%%%%%%%%%%%%%%%%%%%%%
\begin{remark}
It is an interesting open question if one can extend Theorems \ref{InfBandLT} and \ref{InfBandGLT} to also cover the endpoint results $p=1/2$, respectively, $p=1$.
%Whether or not the critical Lieb--Thirring bound, that is, \eqref{InfBand-LT1} with $p=1/2$ or \eqref{InfBand-LT3} with $p=1$, holds for some infinite band set $\E$
%remains an open problem.
In this regard, we point out that $\sum_{k=1}^\infty \eps_k<\infty$ is a necessary condition.
Indeed, let $J'\in\cc T_\E$ be such that the spectral measure $d\rho$ of $(J',\de_0)$ has the form \eqref{InfBand-rp} with $\ga_j=\be_j$ for all $j\geq1$, equivalently,
\begin{align}
\bigg|\int_\E\f{d\rho(t)}{t-\la}\bigg| = \sup_{d\mu\in\cc R_\E}\bigg|\int_\E\f{d\mu(t)}{t-\la}\bigg|,\quad \la<\be_0,
\end{align}
and consider the rank one perturbation $J=J'+\de b_0\langle\de_0,\,\cdot\,\rangle\de_0$. Then, as in \eqref{rank one},
%by the Aronszajn--Donoghue theory of rank one perturbations (see for example \cite[Sect.12.2]{Si05}) we have that
$\la\in\bb R\bsn\E$ is an eigenvalue of $J$ if and only if
\begin{equation}
\int_\E\f{d\rho(t)}{t-\la}=\langle\de_0,(J'-\la)^{-1}\de_0\rangle = -\f{1}{\de b_0}.
\end{equation}
Assume that $\de b_0<0$ and denote by $\la_0$ the eigenvalue of $J$ below $\be_0=\inf\E$. It is known (cf. \cite{To06}) that the Green function satisfies
\begin{equation}
g(x)\geq c|\be_0-x|^{1/2}
\end{equation}
for some $c>0$ and all $x<\be_0$ sufficiently close to $\be_0$. Hence, it follows from \eqref{InfBand-LT2} with $p=1/2$, respectively, \eqref{InfBand-LT3} with $p=1$
that $|\la_0-\be_0|^{1/2}\leq C|\de b_0|$ for some constant $C<\infty$ and all $\de b_0<0$ sufficiently close to zero. Thus,
\begin{align}
\limsup_{\la\nearrow\be_0}|\la-\be_0|^{1/2}\bigg| \int_\E\f{d\rho(t)}{t-\la}\bigg| = \limsup_{\de b_0\nearrow0}\f{|\la_0-\be_0|^{1/2}}{|\de b_0|} \leq C < \infty,
\end{align}
and hence, $\sum_{k=1}^\infty \eps_k<\infty$ follows from the converse direction of Theorem~\ref{InfBandE}.

The above considerations also lead to the new insight that there are
%As a consequence of the above, there turn out to be
several homogeneous sets for which the endpoint Lieb--Thirring bounds (i.e., \eqref{InfBand-LT2} with $p=1/2$, respectively, \eqref{InfBand-LT3} with $p=1$) cannot hold.
For example, every infinite band set of the form \eqref{InfBandSetE} with
\begin{equation}
\sup_{k\geq1}\eps_k<1 \; \mbox{ and } \; \sum_{k=1}^\infty \eps_k = \infty.
\end{equation}
Moreover, we see that the endpoint results do not even need to hold for homogeneous sets with optimally smooth Green function (i.e., H\"older continuous of order $1/2$).
Indeed, in our setting a result of Totik \cite[Corollary 3.3]{To06} implies that $g\in\lip(1/2)$ precisely when $\sum_{k=1}^\infty\eps_k^2<\infty$.
So the infinite band set $\E$ with $\eps_k=1/(k+1)$ is homogeneous and the Green function for $\overline{\bb C}\bsn\E$ is optimally smooth. %is H\"older continuous of order $1/2$.
Yet, the endpoint Lieb--Thirring bounds do not hold for perturbations of some element in $\cc T_\E$.
\end{remark}

%%%%%%%%%%%%%%%%%%%%%%%%%%%%%%%%%%%%%%%
\subsection{$\mathbold{\eps}$-Cantor Set Example}
\label{sec4b}
%%%%%%%%%%%%%%%%%%%%%%%%%%%%%%%%%%%%%%%

In this subsection, we consider fat Cantor sets (i.e., those of positive Lebesgue measure). Suppose $\{\eps_k\}_{k=1}^\infty\subset(0,1)$ and
let
\begin{equation} \lb{CantorSetE}
\E=\bigcap_{k=0}^\infty\E_k
\end{equation}
be the middle $\mathbold{\eps}$-Cantor set, that is, $\E_0=[\be_0,\al_0]$ and $\E_{k}$ is obtained from $\E_{k-1}$
by removing the middle $\eps_{k}$ portion from each of the $2^{k-1}$ bands in $\E_{k-1}$. It is known (cf.\ \cite[p.~125]{PY03}) that $\E$ is a homogeneous set
(in particular, $\E$ is of positive measure) if and only if $\sum_{k=1}^\infty \eps_k<\infty$.
%We shall always assume that this condition holds true.

Our first main result is
%%%%%%%%%%%%%%%%%%%%%%%%%%%%%%%%%%%%%%%
\begin{theorem} \lb{CantorE}
Suppose $\E$ is the middle $\mathbold{\eps}$-Cantor set constructed in \eqref{CantorSetE}. If $\sum_{k=1}^\infty k\eps_k<\infty$,
then for some constant $C>0$,
\begin{align}
 \label{gap}
\sup_{d\rho\in\cc R_\E}\bigg|\int_\E\f{d\rho(t)}{t-x}\bigg| \leq
\begin{cases}
\ds{\f{C}{|x-\al_0|^{1/2}|x-\be_0|^{1/2}}, \quad x\in\bb R\bsn\E_0,} \\[5mm]
\ds{\qquad \f{C\sqrt{\eps_k}}{\dist(x,\E)^{1/2}}, \quad x\in\E_{k-1}\bsn\E_k,\quad k\geq1.}
\end{cases}
\end{align}
Conversely, if
\begin{align}
\lb{outer sqrt left}
\limsup_{x\nearrow\be_0}\, |x-\be_0|^{1/2}\!\sup_{d\rho\in\cc R_\E}\bigg|\int_\E\f{d\rho(t)}{t-x}\bigg| < \infty,
\end{align}
then $\sum_{k=1}^\infty k\eps_k<\infty$.
\end{theorem}
%%%%%%%%%%%%%%%%%%%%%%%%%%%%%%%%%%%%%%%
\begin{remark}
By symmetry, the condition in \eqref{outer sqrt left} is equivalent to
\begin{align}
\lb{outer sqrt right}
\limsup_{x\searrow\,\al_0}\, |x-\al_0|^{1/2}\!\sup_{d\rho\in\cc R_\E}\bigg|\int_\E\f{d\rho(t)}{t-x}\bigg| < \infty.
\end{align}
\end{remark}
%%%%%%%%%%%%%%%%%%%%%%%%%%%%%%%%%%%%%%%
\begin{proof}
Assume that $\sum_{k=1}^\infty k\eps_k<\infty$.
Since the first inequality in \eqref{gap} follows directly from Lemma \ref{lem1a} below (with $i=0$ and $m=0$), we merely focus on establishing the estimate
for the inner gaps.
As for notation, denote by $(\al_j,\be_j)$, $j\geq 0$, the gaps of $\E$ and let $\ga_j$ be an arbitrary point in $[\al_j,\be_j]$ for $j\geq 1$.
Moreover, let
\begin{equation}
b_k=\f{(1-\eps_1)\cdots(1-\eps_k)(\al_0-\be_0)}{2^k}, \quad k\geq0,
\end{equation}
and
\begin{equation}
g_k=\f{\eps_k(1-\eps_1)\cdots(1-\eps_{k-1})(\al_0-\be_0)}{2^{k-1}}, \quad k\geq1,
\end{equation}
be the band and gap lengths at level $k$.
Fix a gap, say $(\al_{j_k},\be_{j_k})$, at level $k\geq 1$ (i.e., an interval in $\E_{k-1}\bsn\E_k$).
We claim that it suffices to show that
\begin{equation}
\f{1}{\sqrt{|x-\be_0| |x-\al_0|}}\prod_{j\neq j_k}\f{|x-\ga_j|}{\sqrt{|x-\al_j| |x-\be_j|}}
\leq \f{C}{\sqrt{b_k}}
\end{equation}
when $x\in(\al_{j_k},\be_{j_k})$. For it readily follows that
\begin{equation}
\f{|x-\ga_{j_k}|}{\sqrt{|x-\al_{j_k}| |x-\be_{j_k}|}}\leq
\f{\sqrt{g_k}}{\dist(x,\E)^{1/2}}
\end{equation}
and
\begin{equation}
\label{gap band ratio}
\f{g_k}{b_k}=\f{2\eps_k}{1-\eps_k}.
\end{equation}

Suppose that $x\in(\al_{j_k},\be_{j_k})$ and set $B_0=\E_0$.
If $k>1$, then $x$ belongs to precisely one of the two bands in $\E_1$. Denote this band by $B_1$.
Similarly, if $k>2$, denote by $B_2$ the unique band in $\E_2\cap B_1$ which contains $x$. We may continue in this way to
obtain a finite sequence of bands
\begin{equation}
B_0\supset B_1\supset B_2\supset\ldots\supset B_{k-1},
\end{equation}
each of which contains $x$. As for further notation, let $(\al_{j_i},\be_{j_i})$ denote the gap in $\E_i\cap B_{i-1}$ for $i=1,\ldots,k-1$.
Note that $(\al_{j_k},\be_{j_k})$ precisely matches the gap in $\E_k\cap B_{k-1}$. A possible scenario when $k=4$ is illustrated below.

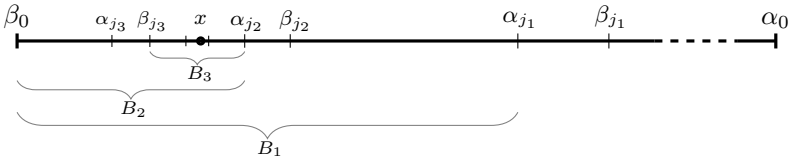
\begin{center}
\begin{tikzpicture}
\draw [very thick] (-5,0) node[align=center, above=1pt] {$\be_0$} -- (3.4,0);
%\draw [dashed, very thick] (-4.4,0) -- (-4.1,0);
\draw [very thick] (4.5,0) -- (5,0) node[align=center, above=1pt] {$\al_0$};
\draw [dashed, very thick] (3.4,0) -- (4.5,0);
%\draw [very thick] (-4,0) -- (4,0);
\draw [very thick] (-5,-0.1) -- (-5,0.1);
\draw [very thick] (5,-0.1) -- (5,0.1);
\draw [very thick](-2.58,0) circle [radius=0.04] node[align=center, above=2.5pt] {\footnotesize $x$};
%\draw [thick] (-3,-0.1) -- (-3,0.1);
%\draw [thick] (3,-0.1) -- (3,0.1);
%\draw [gray,decorate,decoration={brace,amplitude=10pt},yshift=21pt] (-3,0)  -- (3,0)
%   node [black,midway,above=5pt,xshift=1.5pt] {$A_i$};
\draw [gray,decorate,decoration={brace,amplitude=10pt,mirror},yshift=-27pt] (-5,0)  -- (1.6,0)
   node [black,midway,below=7.5pt,xshift=1pt] {\scriptsize $B_1$};
%\draw [gray,decorate,decoration={brace,amplitude=4pt},yshift=5pt] (-0.5,0)  -- (0.5,0)
%   node [black,midway,above=1.5pt,xshift=1pt] {\scriptsize $G_1$};
\draw [thin] (1.6,-0.1) -- (1.6,0.1) node[align=center, above=-1.3pt, xshift=1pt] {\small $\al_{j_1}$};
\draw [thin] (2.8,-0.1) -- (2.8,0.1) node[align=center, above=-1.3pt, xshift=1pt] {\small $\be_{j_1}$};
\draw [gray,decorate,decoration={brace,amplitude=7pt,mirror},yshift=-15pt] (-5,0)  -- (-2,0)
   node [black,midway,below=4.5pt,xshift=1pt] {\scriptsize $B_2$};
%\draw [gray,decorate,decoration={brace,amplitude=2pt},yshift=5pt] (-2,0)  -- (-1.5,0)
%   node [black,midway,above=-0.5pt,xshift=2pt] {\scriptsize $G_2$};
\draw [thin] (-2,-0.09) -- (-2,0.09) node[align=center, above=-2.3pt, xshift=1pt] {\footnotesize $\al_{j_2}$};
\draw [thin] (-1.4,-0.09) -- (-1.4,0.09) node[align=center, above=-2.3pt, xshift=2pt] {\footnotesize $\be_{j_2}$};
\draw [gray,decorate,decoration={brace,amplitude=5pt,mirror},yshift=-4pt] (-3.25,0)  -- (-2,0)
   node [black,midway,below=2pt,xshift=1pt] {\scriptsize $B_3$};
%\draw [gray,decorate,decoration={brace,amplitude=1pt},yshift=5pt] (-2.65,0)  -- (-2.35,0)
%   node [black,midway,above=-1pt,xshift=-3pt] {\scriptsize $\cdot\hspace{-1pt}\cdot G_3$};
\draw [thin] (-3.75,-0.08) -- (-3.75,0.08) node[align=center, above=-2pt] {\scriptsize $\al_{j_3}$};
\draw [thin] (-3.25,-0.08) -- (-3.25,0.08) node[align=center, above=-2pt, xshift=1pt] {\scriptsize $\be_{j_3}$};
%\draw [gray,decorate,decoration={brace},xshift=-4pt,yshift=-9pt] (-1,0)  -- (1,0)
%   node [black,midway,above=4pt,xshift=-2pt] {\footnotesize $G_1$};
\draw [thin] (-2.775,-0.07) -- (-2.775,0.07); % node[align=center, below=4pt, xshift=-1pt] {\tiny $\al_{j_4}$};
\draw [thin] (-2.475,-0.07) -- (-2.475,0.07); % node[align=center, below=3pt, xshift=4pt] {\tiny $\be_{j_4}$};
%\draw [thin] (-4.75,-0.07) -- (-4.75,0.07);
%\draw [thin] (-4.65,-0.07) -- (-4.65,0.07);
\end{tikzpicture}
\end{center}
We observe that $B_i$ and $B_{i+1}$ always have precisely one endpoint in common.

Our estimation now splits into three parts. We start by estimating the product corresponding to all the gaps of $\E$ which are contained in
$(\E_i\cap B_{i-1})\bsn B_i$ for $i=1,\ldots,k-1$. As follows from Lemma \ref{lem1b}, this infinite product is bounded as long as $\sum_{k=1}^\infty k\eps_k<\infty$.
%$\sum_{k=1}^\infty\eps_k<\infty$.
Then we estimate the finite product corresponding to the endpoints $\al_0$, $\be_0$ and the gaps $(\al_{j_i},\be_{j_i})$ for $i=1,\ldots,k-1$.
This product is bounded by some constant divided by $\sqrt{b_k}$, see Lemma \ref{lem2} below.
The final step is to estimate the product corresponding to the gaps in $B_{k-1}\bsn(\al_{j_k},\be_{j_k})$. But this can be done as in
Lemma \ref{lem1a} (with $i=k$ and $m=0$).

For the converse direction, we mimic the proof of Theorem \ref{InfBandE} and take $d\rho$ to be the reflectionless measure on $\E$
which corresponds to $\ga_j=\be_j$ for all $j\geq 1$.
It then suffices to show that
\begin{equation}
\label{prod to sum}
\prod_{j=1}^\infty \frac{\be_j-\be_0}{\al_j-\be_0}<\infty
\;\; \Longrightarrow \;\; \sum_{k=1}^\infty k\eps_k<\infty.
\end{equation}
Convergence of the above product implies that the factors are bounded. Hence,
\begin{equation}
\frac{\be_j-\be_0}{\al_j-\be_0}=1+\frac{\be_j-\al_j}{\al_j-\be_0}\geq \exp\biggl\{\frac{1}{d}\frac{\be_j-\al_j}{\al_j-\be_0}\biggr\}
\end{equation}
for some constant $d>0$ and all $j\geq 1$. Our aim is thus to show that
\begin{equation}
\label{sum j}
 \sum_{j=1}^\infty\frac{\be_j-\al_j}{\al_j-\be_0}\geq c\sum_{k=1}^\infty k\eps_{k}
\end{equation}
for some constant $c>0$. This will immediately imply \eqref{prod to sum}.
For the sake of clarity, we shall refer to the following figure.
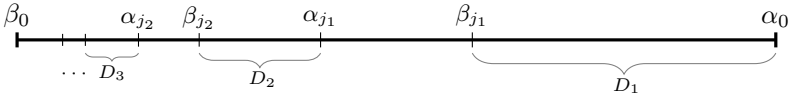
\begin{center}
\begin{tikzpicture}
\draw [very thick] (-5,0) node[align=center, above=1pt] {$\be_0$} -- (5,0) node[align=center, above=1pt] {$\al_0$};
%\draw [dashed, very thick] (-4.4,0) -- (-4.1,0);
%\draw [very thick] (4.5,0) -- (5,0) node[align=center, below=2pt] {$\al_0$};
%\draw [dashed, very thick] (4.1,0) -- (4.4,0);
%\draw [very thick] (-4,0) -- (4,0);
\draw [very thick] (-5,-0.1) -- (-5,0.1);
\draw [very thick] (5,-0.1) -- (5,0.1);
%\draw [very thick](-3.7,0) circle [radius=0.05] node[align=center, above] {$x$};
%\draw [thick] (-3,-0.1) -- (-3,0.1);
%\draw [thick] (3,-0.1) -- (3,0.1);
%\draw [gray,decorate,decoration={brace,amplitude=10pt},yshift=21pt] (-3,0)  -- (3,0)
%   node [black,midway,above=5pt,xshift=1.5pt] {$A_i$};
\draw [gray,decorate,decoration={brace,amplitude=8pt,mirror},yshift=-5pt] (1,0)  -- (5,0)
   node [black,midway,below=6pt,xshift=1pt] {\scriptsize $D_1$};
%\draw [gray,decorate,decoration={brace,amplitude=4pt},yshift=5pt] (-0.5,0)  -- (0.5,0)
%   node [black,midway,above=1.5pt,xshift=1pt] {\scriptsize $G_1$};
\draw [thin] (-1,-0.1) -- (-1,0.1) node[align=center, above=-1.3pt] {\small $\al_{j_1}$};
\draw [thin] (1,-0.1) -- (1,0.1) node[align=center, above=-1.3pt] {\small $\be_{j_1}$};
\draw [gray,decorate,decoration={brace,amplitude=5.5pt,mirror},yshift=-5pt] (-2.6,0)  -- (-1,0)
   node [black,midway,below=3.5pt,xshift=1pt] {\scriptsize $D_2$};
%\draw [gray,decorate,decoration={brace,amplitude=2pt},yshift=5pt] (-2,0)  -- (-1.5,0)
%   node [black,midway,above=-0.5pt,xshift=2pt] {\scriptsize $G_2$};
\draw [thin] (-3.4,-0.09) -- (-3.4,0.09) node[align=center, above=-1.8pt] {\small $\al_{j_2}$};
\draw [thin] (-2.6,-0.09) -- (-2.6,0.09) node[align=center, above=-1.8pt] {\small $\be_{j_2}$};
\draw [gray,decorate,decoration={brace,amplitude=3pt,mirror},yshift=-5pt] (-4.1,0)  -- (-3.4,0)
   node [black,midway,below=1pt,xshift=-7pt] {\scriptsize $\cdots$ $D_3$};
%\draw [gray,decorate,decoration={brace,amplitude=1pt},yshift=5pt] (-2.65,0)  -- (-2.35,0)
%   node [black,midway,above=-1pt,xshift=-3pt] {\scriptsize $\cdot\hspace{-1pt}\cdot G_3$};
\draw [thin] (-4.4,-0.08) -- (-4.4,0.08); %node[align=center, above=-2pt] {\tiny $\al_3$};
\draw [thin] (-4.1,-0.08) -- (-4.1,0.08); %node[align=center, above=-2pt] {\tiny $\be_3$};
%\draw [gray,decorate,decoration={brace},xshift=-4pt,yshift=-9pt] (-1,0)  -- (1,0)
%   node [black,midway,above=4pt,xshift=-2pt] {\footnotesize $G_1$};
%\draw [thin] (-4.75,-0.07) -- (-4.75,0.07);
%\draw [thin] (-4.65,-0.07) -- (-4.65,0.07);
\end{tikzpicture}
\end{center}
The idea is to estimate the terms from all the gaps in $D_1$, all the gaps in $D_2$, etc.,
as well as the term from the gap between $D_1$ and $D_2$, the gap between $D_2$ and $D_3$, etc.
Start by noting that
\begin{align}
 \sum_{j:\,(\al_j, \be_j)\subset D_n}\frac{\be_j-\al_j}{\al_j-\be_0} &
  \geq \frac{1}{b_{n-1}} \Bigl(g_{n+1}+2g_{n+2}+\ldots+2^{k-1}g_{n+k}+\ldots\Bigr) \no \\
 & = \frac{(1-\eps_n)\eps_{n+1}}{2}+\frac{(1-\eps_n)(1-\eps_{n+1})\eps_{n+2}}{2}+\ldots \no \\
 & \geq \frac{1}{2}\,\prod_{i=n}^\infty (1-\eps_i)\sum_{k=n+1}^\infty \eps_k
\end{align}
for every $n\geq 1$. If $(\al_{j_n},\be_{j_n})$ denotes the gap between $D_n$ and $D_{n+1}$, it follows from \eqref{gap band ratio} that
\begin{equation}
 \sum_{n=1}^\infty\frac{\be_{j_n}-\al_{j_n}}{\al_{j_n}-\be_0}=
 \sum_{k=1}^\infty \frac{g_k}{b_k}\geq 2\sum_{k=1}^\infty \eps_k.
\end{equation}
Hence, we obtain \eqref{sum j} with $2c=\prod_{i=1}^\infty (1-\eps_i)>0$. This completes the proof.
\end{proof}
%%%%%%%%%%%%%%%%%%%%%%%%%%%%%%%%%%%%%%%
We now formulate and prove the three technical lemmas that are needed in the proof of Theorem \ref{CantorE}.
%%%%%%%%%%%%%%%%%%%%%%%%%%%%%%%%%%%%%%%
\begin{lemma} \lb{lem1a}
Suppose $\sum_{j=1}^\infty j\eps_j<\infty$ and consider the infinite products
\begin{equation}
 R_i(x)=\prod_{j:\,(\al_j,\be_j)\subset A_i}\f{|x-\ga_j|}{\sqrt{|x-\al_j| |x-\be_j|}},
 \quad i=0,1,\ldots,k,
\end{equation}
where $A_i$ is a band in $\E_i$. When $\dist(x, A_i)\geq mb_i$, we have
\begin{equation} \lb{Ri}
 R_i(x)\leq\exp\biggl\{\frac{1}{c}\sum_{j=i+1}^\infty\Bigl(\sum_{n=1}^{j-i}\f{1}{1+m2^n}\Bigr)\eps_{j}\biggr\},
\end{equation}
where $c=\prod_{j=1}^\infty(1-\eps_j)$.
\end{lemma}
%%%%%%%%%%%%%%%%%%%%%%%%%%%%%%%%%%%%%%%
\begin{proof}
Let us assume that the point $x$ lies to the left of the band $A_i$.
Then
\begin{equation}
 R_i(x)\leq%\prod_j\frac{\be_j-x}{\sqrt{(\al_j-x)(\be_j-x)}} \\
 \prod_{j:\,(\al_j,\be_j)\subset A_i}\sqrt{1+\f{\be_j-\al_j}{\al_j-x}}
 \leq\exp\biggl\{\frac{1}{2}\sum_{j:\,(\al_j,\be_j)\subset A_i}\f{\be_j-\al_j}{\al_j-x}\biggr\}.
\end{equation}
With the figure below in mind, the idea is for every $n\geq 1$ to estimate the term from the gap $G_n$ and the terms from all the gaps in $F_n$.
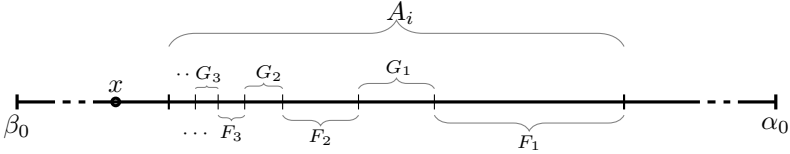
\begin{center}
\begin{tikzpicture}
\draw [very thick] (-5,0) node[align=center, below=1pt] {$\be_0$} -- (-4.5,0);
\draw [dashed, very thick] (-4.4,0) -- (-4.1,0);
\draw [very thick] (4.5,0) -- (5,0) node[align=center, below=2pt] {$\al_0$};
\draw [dashed, very thick] (4.1,0) -- (4.4,0);
\draw [very thick] (-4,0) -- (4,0);
\draw [very thick] (-5,-0.1) -- (-5,0.1);
\draw [very thick] (5,-0.1) -- (5,0.1);
\draw [very thick](-3.7,0) circle [radius=0.05] node[align=center, above] {$x$};
\draw [thick] (-3,-0.1) -- (-3,0.1);
\draw [thick] (3,-0.1) -- (3,0.1);
\draw [gray,decorate,decoration={brace,amplitude=10pt},yshift=21pt] (-3,0)  -- (3,0)
   node [black,midway,above=5pt,xshift=1.5pt] {$A_i$};
\draw [gray,decorate,decoration={brace,amplitude=6pt,mirror},yshift=-5pt] (0.5,0)  -- (3,0)
   node [black,midway,below=4pt,xshift=-0pt] {\scriptsize $F_1$};
\draw [gray,decorate,decoration={brace,amplitude=4pt},yshift=5pt] (-0.5,0)  -- (0.5,0)
   node [black,midway,above=1.5pt,xshift=1pt] {\scriptsize $G_1$};
\draw [thin] (-0.5,-0.1) -- (-0.5,0.1);
\draw [thin] (0.5,-0.1) -- (0.5,0.1);
\draw [gray,decorate,decoration={brace,amplitude=4pt,mirror},yshift=-5pt] (-1.5,0)  -- (-0.5,0)
   node [black,midway,below=2pt,xshift=-0pt] {\scriptsize $F_2$};
\draw [gray,decorate,decoration={brace,amplitude=2pt},yshift=5pt] (-2,0)  -- (-1.5,0)
   node [black,midway,above=-0.5pt,xshift=2pt] {\scriptsize $G_2$};
\draw [thin] (-2,-0.1) -- (-2,0.1);
\draw [thin] (-1.5,-0.1) -- (-1.5,0.1);
\draw [gray,decorate,decoration={brace,amplitude=2pt,mirror},yshift=-5pt] (-2.35,0)  -- (-2,0)
   node [black,midway,below=0.5pt,xshift=-7pt] {\scriptsize $\cdots$ $F_3$};
\draw [gray,decorate,decoration={brace,amplitude=1pt},yshift=5pt] (-2.65,0)  -- (-2.35,0)
   node [black,midway,above=-1pt,xshift=-3pt] {\scriptsize $\cdot\hspace{-1pt}\cdot G_3$};
\draw [thin] (-2.65,-0.1) -- (-2.65,0.1);
\draw [thin] (-2.35,-0.1) -- (-2.35,0.1);
%\draw [gray,decorate,decoration={brace},xshift=-4pt,yshift=-9pt] (-1,0)  -- (1,0)
%   node [black,midway,above=4pt,xshift=-2pt] {\footnotesize $G_1$};
\end{tikzpicture}
\end{center}
If $\dist(x, A_i)\geq mb_i$ and $G_n=(\al_n,\be_n)$, we have
\begin{equation}
 \f{\be_n-\al_n}{\al_n-x}\leq\f{g_{i+n}}{b_{i+n}+mb_i}\leq\f{2\eps_{i+n}}{c+m2^n}
\end{equation}
and
\begin{align}
 \notag
 \sum_{j:\,(\al_j,\be_j)\subset F_n}\f{\be_j-\al_j}{\al_j-x}
 &\leq\frac{1}{b_{i+n}+mb_i}\sum_{j>n} 2^{j-n-1} g_{i+j} \\
 &\leq\frac{1}{c(1+m2^n)}\sum_{j>n}\eps_{i+j}.
\end{align}
It hence follows that
\begin{equation}
R_i(x)\leq\exp\biggl\{\frac{1}{c}\sum_{n=1}^\infty\Bigl(\f{1}{1+m2^n}\sum_{j\geq n}\eps_{i+j}\Bigr)\biggr\}
% &=\exp\biggl\{\f{1}{c}\sum_{k>i}\Bigl(\sum_{n=1}^{k-i}\f{1}{1+m2^n}\Bigr)\eps_k\biggr\},
\end{equation}
and \eqref{Ri} is obtained by interchanging the order of summation.
\end{proof}
%%%%%%%%%%%%%%%%%%%%%%%%%%%%%%%%%%%%%%%

%%%%%%%%%%%%%%%%%%%%%%%%%%%%%%%%%%%%%%%
\begin{lemma} \lb{lem1b}
Suppose $\sum_{j=1}^\infty j\eps_j<\infty$ and let $A$ denote the set given by
\begin{equation}
A=\bigcup_{i=1}^{k-1} (\E_i\cap B_{i-1})\bsn B_i.
\end{equation}
When $x\in(\al_{j_k},\be_{j_k})$, we have
\begin{equation}  \lb{A}
 \prod_{j:\,(\al_j,\be_j)\subset A} \f{|x-\ga_j|}{\sqrt{|x-\al_j| |x-\be_j|}}
 \leq\exp\biggl\{\f{2}{c}\sum_{j=2}^\infty(j-1)\eps_j\biggr\},
\end{equation}
where $c=\prod_{j=1}^\infty(1-\eps_j)$.
\end{lemma}
%%%%%%%%%%%%%%%%%%%%%%%%%%%%%%%%%%%%%%%
\begin{proof}
The set $A$ is the union of $2^{k-1}-1$ bands in $\E_{k-1}$ (namely all bands except for $B_{k-1}$) and $2^{i-1}-1$ gaps at level $i$ for $i=2,\ldots,k-1$.
Let $F_1, F_2, \ldots, F_{2^{k-1}-1}$ be an ordering of the bands in $\E_{k-1}\bsn B_{k-1}$ so that
\begin{equation}
 \dist(x,F_1)\leq\ldots\leq\dist(x,F_{2^{k-1}-1}).
\end{equation}
By construction,
\begin{equation}
 \dist(x,F_{2m+1})\geq mb_{k-1} \; \mbox{ for } \; m=0,1,\ldots,2^{k-2}-1,
\end{equation}
and since
\begin{equation}
 \sum_{n=1}^{j+1-k}\frac{1}{1+m2^n}\leq\f{1}{2^i} \; \mbox{ when } \; m\geq 2^i,
\end{equation}
we have
\begin{equation}
 \sum_{m=0}^{2^{k-2}-1}\Bigl(\sum_{n=1}^{j+1-k}\f{1}{1+m2^n}\Bigr)\leq
 j+1-k+\sum_{i=0}^{k-3} \Bigl(2^i\cdot\frac{1}{2^i}\Bigr)=j-1.
\end{equation}
Here, the term $j+1-k$ comes from $m=0$ and the inner sum is bounded by $1/2^i$ for the $2^i$ terms corresponding to $m=2^i, \ldots, 2^{i+1}-1$.
When $i$ runs from $0$ to $k-3$, we get the entire sum for $m\geq 1$. By Lemma \ref{lem1a}, it follows that
\begin{equation}  \lb{est sum}
 \prod_{j:\,(\al_j,\be_j)\subset\E_{k-1}\bs B_{k-1} } \f{|x-\ga_j|}{\sqrt{|x-\al_j| |x-\be_j|}}
 \leq\exp\biggl\{\frac{2}{c}\sum_{j=k}^\infty (j-1)\eps_j\biggr\}.
\end{equation}

To finish the proof, fix a level $i\in\{2,\ldots,k-1\}$ and order the $2^{i-1}-1$ gaps at this level according to their distance to $x$.
The $m$th gap in this ordering, say $G_m=(\al_m,\be_m)$, then satisfies that
\begin{equation}
 \dist(x,G_m)\geq mb_i.
\end{equation}
Since
\begin{equation}
 \f{|x-\ga_m|}{\sqrt{|x-\al_m||x-\be_m|}}\leq\sqrt{1+\f{g_i}{mb_i}}\leq\sqrt{1+\f{2\eps_i}{cm}}
\end{equation}
and
\begin{equation}
 \sum_{m=1}^{2^{i-1}-1}\f{1}{m}
 \leq 1+\biggl(\frac{1}{2}+\frac{1}{2}\biggr)+\ldots+\biggl(\frac{1}{2^{i-2}}+\ldots+\f{1}{2^{i-2}}\biggr)=i-1,
\end{equation}
it follows that
\begin{equation}  \lb{est i}
 \prod_{m=1}^{2^{i-1}-1}\f{|x-\ga_m|}{\sqrt{|x-\al_m| |x-\be_m|}}
 \leq\exp\biggl\{\frac{1}{c} (i-1)\eps_i\biggr\}.
\end{equation}
The proof of \eqref{A} is now an immediate consequence of \eqref{est sum} and \eqref{est i} for $i=2,\ldots,k-1$ .
\end{proof}
%%%%%%%%%%%%%%%%%%%%%%%%%%%%%%%%%%%%%%%

%%%%%%%%%%%%%%%%%%%%%%%%%%%%%%%%%%%%%%%
\begin{lemma} \lb{lem2}
Suppose $\sum_{j=1}^\infty \eps_j<\infty$ and consider the finite product
\begin{equation}
Q(x)=\f{1}{\sqrt{|x-\be_0||x-\al_0|}}\prod_{i=1}^{k-1}\f{|x-\ga_{j_i}|}{\sqrt{|x-\al_{j_i}| |x-\be_{j_i}|}}.
\end{equation}
When $x\in(\al_{j_k},\be_{j_k})$, we have
\begin{equation}
Q(x)\leq \sqrt{\frac{2}{\al_0-\be_0}}
\exp\biggl\{\frac{2}{c}\sum_{i=1}^k\eps_i\biggr\} \f{1}{\sqrt{b_k}},
\end{equation}
where $c=\prod_{j=1}^\infty (1-\eps_j)$.
\end{lemma}
%%%%%%%%%%%%%%%%%%%%%%%%%%%%%%%%%%%%%%%
\begin{proof}
As in Theorem \ref{ReflEstThm}, we pick $\tilde{\ga}_{j_i}\in\{\al_{j_i},\be_{j_i}\}$ so that
\begin{equation}
|x-\tilde{\ga}_{j_i}|=\max\{|x-\al_{j_i}|, |x-\be_{j_i}|\}, \quad i=1,\ldots,k-1.
\end{equation}
The other point in $\{\al_{j_i},\be_{j_i}\}$ will be denoted by $\bar{\ga}_{j_i}$.
Since $|x-\bar{\ga}_{j_{k-1}}|\geq b_k$, it follows directly that
\begin{align} \lb{Q}
\notag
Q(x)&\leq\f{1}{\sqrt{|x-\be_0| |x-\al_0|}}\prod_{i=1}^{k-1}\sqrt{\f{|x-\tilde{\ga}_{j_i}|}{|x-\bar{\ga}_{j_i}|}} \\
&\leq\sqrt{\f{2}{\al_0-\be_0}}\,\prod_{i=1}^{k-1}\sqrt{\f{|x-\tilde{\ga}_{j_i}|}{|x-\bar{\ga}_{j_{i-1}}|}}\,
\f{1}{\sqrt{b_k}},
\end{align}
where $\bar{\ga}_{j_0}\in\{\al_0,\be_0\}$ is chosen so that
\begin{equation}
|x-\bar{\ga}_{j_0}|=\min\{|x-\al_0|, |x-\be_0|\}.
\end{equation}
Note that $\bar{\ga}_{j_0}, \bar{\ga}_{j_1}, \ldots, \bar{\ga}_{j_{k-1}}$ coincide with the endpoints of $B_1, \ldots, B_{k-1}$ (counting the common endpoints only once). The ordering, however, can be arbitrary.

In order to estimate the product over $i$ in \eqref{Q}, we rearrange the factors in the denominator.
Let $\bar{\ga}_{j_{\si(i)}}$ be the endpoint of $B_i$ which is farthest from $x$
(this happens to be the endpoint of $B_i$ which is {\it not} an endpoint of $B_{i+1}$).
Then $x$ is closer to the other endpoint of $B_i$ and we have
%Divide $|x-\tilde{\ga}_{j_1}|$ by $|x-\si_1|$,
%where $\si_1\in M_1=\{\bar{\ga}_{j_0},\bar{\ga}_{j_1}\}$ is the point for which
%\begin{equation}
%|x-\si_1|=\max_{\si\in M_1}\{|x-\si|\}.
%\end{equation}
%Then divide $|x-\tilde{\ga}_{j_2}|$ by $|x-\si_2|$, where $\si_2\in M_2=\{\bar{\ga}_{j_0},\bar{\ga}_{j_1},\bar{\ga}_{j_2}\}\bsn\{\si_1\}$
%is chosen so that
%\begin{equation}
%|x-\si_2|=\max_{\si\in M_2} \{|x-\si|\}.
%\end{equation}
%Continue in this way and finally divide $|x-\tilde{\ga}_{j_{k-2}}|$ by $|x-\si_{k-2}|$, where
%$\si_{k-2}\in M_{k-2}=\{\bar{\ga}_{j_0},\ldots,\bar{\ga}_{j_{k-2}}\}\bsn\{\si_1,\ldots,\si_{k-3}\}$
%is the point with
%\begin{equation}
%|x-\si_{k-2}|=\max_{\si\in M_{k-2}} \{|x-\si|\}.
%\end{equation}
\begin{equation} \lb{prod i}
\f{|x-\tilde{\ga}_{j_i}|}{|x-\bar{\ga}_{j_{\si(i)}}|}\leq\f{g_i+b_i/2}{b_i/2}\leq 1+\f{4\eps_i}{c}
\end{equation}
for $i=1,\ldots,k-2$.
%To see this, note that each of the bands $B_i$ lies between a gap at level $i$, say $G_i$, and some larger gap, $H_i$.
%If $x$ is closer to $H_i$ than $G_i$, then $\si_i=\bar{\ga}_{j_i}$. Otherwise, $\si_i$ is the endpoint of $H_i$ which also belongs to $B_i$.
%What is left is to estimate ${|x-\tilde{\ga}_{j_{k-1}}|}/{|x-\si_{k-1}|}$, where $\si_{k-1}$ is the point in
%$\{\bar{\ga}_{j_0},\ldots,\bar{\ga}_{j_{k-2}}\}$ which differs from $\si_1,\ldots,\si_{k-2}$.
Since $\bar{\ga}_{j_{\si(k-1)}}$ is an endpoint of $B_{k-1}$, we also have
\begin{equation} \lb{factor k-1}
\f{|x-\tilde{\ga}_{j_{k-1}}|}{|x-\bar{\ga}_{j_{\si(k-1)}}|}\leq
\f{g_{k-1}+b_k+g_k}{b_k}\leq 1+\f{4\eps_{k-1}}{c}+\f{2\eps_k}{c}.
\end{equation}
Hence,
\begin{equation}
 \prod_{i=1}^{k-1}\f{|x-\tilde{\ga}_{j_i}|}{|x-\bar{\ga}_{j_{i-1}}|}=
 \prod_{i=1}^{k-1}\f{|x-\tilde{\ga}_{j_i}|}{|x-\bar{\ga}_{j_{\si(i)}}|}
 \leq \prod_{i=1}^{k}\Bigl(1+\f{4\eps_i}{c}\Bigr),
\end{equation}
and the result follows from \eqref{Q}.
\end{proof}
%%%%%%%%%%%%%%%%%%%%%%%%%%%%%%%%%%%%%%%

As a direct consequence of Theorems~\ref{CantorE}, \ref{ReflEstThm}, \ref{LTthm1}, and \ref{LTthm2}, we get the following Lieb--Thirring bounds.
%%%%%%%%%%%%%%%%%%%%%%%%%%%%%%%%%%%%%%%
\begin{theorem} \lb{CantorLT}
Let $\E$ be the middle $\mathbold{\eps}$-Cantor set constructed in \eqref{CantorSetE} and suppose $J$, $J'$ are two-sided Jacobi matrices such that $J'\in\cc T_\E$ and $J=J'+\de J$ is a compact perturbation of $J'$.
If $\eps_k\leq C/4^{k}$ for some $C>0$ and all large $k$, then
\begin{align} \lb{Cantor-LT1}
\sum_{\la\in\si(J)\bs\E} \dist(\la,\E)^p \leq L_{p,\,\E}\sum_{n\in\bb Z}|\de a_n|+|\de b_n|
\end{align}
for $1/2<p<1$.
If $\eps_k\leq C/a^{k}$ for some $a>4$, $C>0$, and all large $k$, then
\begin{align} \lb{Cantor-LT2}
\sum_{\la\in\si(J)\bs\E} \dist(\la,\E)^{p} \leq  L_{p,\,\E}\sum_{n\in\bb Z} |\de a_n|^{p+1/2}+|\de b_n|^{p+1/2}
\end{align}
for all $p>1/2$.
In either case, the constant $L_{p,\,\E}$ is independent of $J$ and $J'$ and only depends on $p$ and $\E$.
\end{theorem}
%%%%%%%%%%%%%%%%%%%%%%%%%%%%%%%%%%%%%%%
\begin{proof}
By construction of $\E$, we have
\begin{align}
g_k \geq 2^{1-k}\eps_k (\al_0-\be_0)\prod_{j=1}^\infty (1-\eps_j).
\end{align}
So \eqref{gap} combined with \eqref{ReflEst1} yields \eqref{RnCond1} and \eqref{RnCond2} for all the gaps at level $k\geq1$
with a constant
\begin{equation}
C_k=Ck\sqrt{\eps_k}\log(1/\eps_k).
\end{equation}
Here, $C>0$ is sufficiently large and independent of $k$.
Since there are $2^{k-1}$ gaps at level $k$, each of length $g_k\leq2^{1-k}\eps_k (\al_0-\be_0)$, the exponential decay assumptions on $\eps_k$ yield \eqref{CkCond1} and \eqref{CkCond2}. The result now follows from
Theorems~\ref{LTthm1} and \ref{LTthm2}.
\end{proof}
%%%%%%%%%%%%%%%%%%%%%%%%%%%%%%%%%%%%%%%

%%%%%%%%%%%%%%%%%%%%%%%%%%%%%%%%%%%%%%%
As before, let $g$ denote the potential theoretic Green function for the domain $\overline{\bb C}\bsn\E$ with logarithmic pole at infinity.
The counterpart of Theorem \ref{InfBandGLT} for middle $\mathbold{\eps}$-Cantor sets reads

%%%%%%%%%%%%%%%%%%%%%%%%%%%%%%%%%%%%%%%
\begin{theorem} \lb{CantorGLT}
Let $\E$ be the middle $\mathbold{\eps}$-Cantor set constructed in \eqref{CantorSetE} and suppose
$J$, $J'$ are two-sided Jacobi matrices such that $J'\in\cc T_\E$ and $J=J'+\de J$ is a compact perturbation of $J'$.
If $\eps_k\leq C/2^{k}$ for some $C>0$ and all large $k$, then for every $p>1$,
\begin{align} \lb{Cantor-LT3}
\sum_{\la\in\si(J)\bs\E} g(\la)^{p} \leq  L_{p,\,\E}\sum_{n\in\bb Z} |\de a_n|^{(p+1)/2}+|\de b_n|^{(p+1)/2},
\end{align}
where the constant $L_{p,\,\E}$ is independent of $J$, $J'$ and only depends on $p$ and $\E$.
\end{theorem}
%%%%%%%%%%%%%%%%%%%%%%%%%%%%%%%%%%%%%%%
\begin{proof}
As in the proof of Theorem~\ref{InfBandGLT} we use \eqref{dg-dmuE} and the fact that the equilibrium measure for $\E$ is reflectionless. Hence, \eqref{gap} combined with
integration over the gaps yields the estimate
\begin{align}
g(x) \leq C\sqrt{\eps_k}\,\dist(x,\E)^{1/2}, \quad x\in\E_{k-1}\bsn\E_k, \quad k\geq0,
\end{align}
where $\E_{-1}=\bb R$ and $\eps_0=1/e$.
Recall now that $\E$ has $2^{k-1}$ gaps at level $k\geq1$, that is, $\E_{k-1}\bsn\E_k$ consists of $2^{k-1}$ identical intervals.
So as in the proofs of Theorems \ref{LTthm2} and \ref{CantorLT}, we obtain
\begin{align}
\sum_{\la\in\si(J)\cap(\E_{k-1}\bs\E_k)} \dist(\la,\E)^{p/2} \leq  2^{k-1}C_k\sum_{n\in\bb Z} |\de a_n|^{(p+1)/2}+|\de b_n|^{(p+1)/2}, %\quad k\geq0,
\end{align}
where $C_k=Ck\sqrt{\eps_k}\log(1/\eps_k)$. Thus, for each $k\geq0$,
\begin{align}
\sum_{\la\in\si(J)\cap(\E_{k-1}\bs\E_k)} g(\la)^p \leq  C2^kk\eps_k^{(p+1)/2}\log(1/\eps_k)\sum_{n\in\bb Z} |\de a_n|^{(p+1)/2}+|\de b_n|^{(p+1)/2}.
\end{align}
Since $p>1$ and $\eps_k$ decays no slower than $C/2^k$, summing over $k$ yields \eqref{Cantor-LT3}.
\end{proof}
%%%%%%%%%%%%%%%%%%%%%%%%%%%%%%%%%%%%%%%

%%%%%%%%%%%%%%%%%%%%%%%%%%%%%%%%%%%%%%%
%%%%%%%%%%%%%%%%%%%%%%%%%%%%%%%%%%%%%%%
%\bibliographystyle{amsart}

%%%%%%%%%%%%%%%%%%%%%%%%%%%%%%%%%%%%%%%
%%%%%%%%%%%%%%%%%%%%%%%%%%%%%%%%%%%%%%%

\end{document}